\documentclass[reqno]{amsart}
\usepackage{graphicx}% Include figure files

\newtheorem{lemma}{Lemma}
\newtheorem{theorem}{Theorem}
\newtheorem{proposition}{Proposition}

\newtheorem{hypothesis}{Hypothesis}
\newtheorem{remark}{Remark}
\newtheorem{definition}{Definition}

\newcommand{\be}{\begin{eqnarray}}
\newcommand{\ee}{\end{eqnarray}}
\newcommand{\bee}{\begin{eqnarray*}}
\newcommand{\eee}{\end{eqnarray*}}
\newcommand{\R}{{\mathbb R}}
\newcommand{\N}{{\mathbb N}}

\newcommand{\C}{{\mathbb C}}

\newcommand{\I}{\mbox {\sc 1}}

\newcommand{\asy}{{\it O}}

\newcommand{\case}[2]{\textstyle{\frac{#1}{#2}}}

\begin{document}

\title [Bifurcation and stability for double-well NLS] {Bifurcation and stability 
for Nonlinear Schr\"odinger equations 
with double well potential in the semiclassical limit}

\author {Reika FUKUIZUMI}

\address {Graduate School of Information Sciences, Tohoku University, Sendai 980-8579, Japan}

\email {fukuizumi@math.is.tohoku.ac.jp}

\author {Andrea SACCHETTI}

\address {Faculty of Sciences, University of Modena e Reggio Emilia, Modena, Italy}

\email {andrea.sacchetti@unimore.it}

\date {\today}

\thanks {One of us (A.S.) is very grateful to Giuseppe Mazzuoccolo for useful discussions about the Budan-Fourier theorem, and to Riccardo Adami and Hynek Kovarik for useful discussions on NLS equations with singular pointwise interactions.}

\begin {abstract} We consider the stationary solutions for a class of Schr\"odinger equations with a symmetric double-well potential and a nonlinear perturbation. \ Here, in the semiclassical limit we prove that the reduction to a finite-mode approximation give the stationary solutions, up to an exponentially small term, and that symmetry-breaking bifurcation occurs at a given value for the strength of the nonlinear term. \ The kind of bifurcation picture only depends on the non-linearity power. \ We then discuss the stability/instability properties of each branch of the stationary solutions. \ Finally, we consider an explicit one-dimensional toy model where the double well potential is given by means of a couple of attractive Dirac's delta pointwise interactions.
\end{abstract}

\maketitle

\section {Introduction} \label {Sec1}

Here, we consider the stationary solutions of the nonlinear Schr\"odinger (hereafter NLS) equations
\be
i \hbar \frac {\partial \psi }{\partial t} = H_0 \psi + \epsilon g(x) |\psi |^{2\sigma}  \psi , \quad  \| \psi (\cdot ,t) \| =1 , \label {Eq1}
\ee
where $\epsilon \in \R$ and $\| \cdot \|$ denotes the $L^2$ norm, 
\be
H_0 = - \frac {\hbar^2}{2m} \Delta + V , \ \  \ \Delta = \sum_{j=1}^d \frac {\partial^2}{\partial x_j^2} \, ,  \label {Eq2}
\ee
is the linear Hamiltonian and $g(x) |\psi |^{2\sigma} $ is a nonlinear perturbation. \ For the sake of definiteness we assume the units such that $2m=1$.

Atomic Bose-Einstein condensates (BECs) are described by means of nonlinear Schr\"odinger equations of the type (\ref {Eq1}) where $H_0$ represents the Hamiltonian of a single trapped atom and the nonlinear term $ |\psi|^{2\sigma} $, $\sigma=1,2,\ldots $, is the $(\sigma+1)$-body contact potential \cite {Kohler}. \ In fact, BECs strongly depend by interatomic forces and the binary coupling term $ |\psi|^2 \psi$ usually represents the dominant nonlinear term and equation (\ref {Eq1}) takes the form of the well-known Gross-Pitaevskii equation \cite {PitStr}.  \ Even if in most of the applications the parameter $\sigma$ takes only integer and positive values, here we take that $\sigma$ can assume non integer values too, as considered in \cite {Smerzi}. \ 
It is worth mentioning also the fact that equation (\ref {Eq1}) with nonlinearity corresponding to the power-law $|\psi |^{2\sigma}$, where the parameter $\sigma$ takes any positive real value,  is used in other contexts, including semiconductors \cite {Mihalache} and nonlinear optics \cite {Christian, Snyder,Zakharov}.

In this paper we consider the case of symmetric potentials $V$ with \emph {double well} shape; the function $g(x)$ is a bounded regular function (in the following we assume, for argument's sake, that $g(x)$ has the same symmetric properties as $V(x)$). 

If the nonlinear term is absent then the linear Hamiltonian $H_0$ has even--parity and odd-parity eigenstates: the $d$-dimensional linear Schr\"odinger equation with a symmetric double well potential has stationary states of a definite even $\varphi_+$ and odd-parity $\varphi_-$, with associate nondegenerate eigenvalues $\lambda_+ < \lambda_-$. 

However, the introduction of a nonlinear term,  which usually models in quantum mechanics an interacting many-particle system, may give rise to asymmetrical states related to spontaneous symmetry breaking phenomenon.

In the \emph {semiclassical limit} and in the \emph {two-level approximation} has been seen \cite {Sacchetti2} that the symmetric/antisymmetric stable stationary state bifurcates when the adimensional nonlinear parameter $\eta$ takes absolute value equal to the critical value
\begin{eqnarray}
\eta^\star = 2^\sigma /\sigma \, . \label {Eq3}
\end{eqnarray}
The parameter $\eta$ is associated with the coupling factor of the nonlinear perturbation by 
\begin{eqnarray}
\eta  = c\epsilon /{\omega} \label {Eq4}
\end{eqnarray}
and it is the effective nonlinear coupling factor, where $\omega$ is the (half of the) splitting between the two levels 
\begin{eqnarray}
\omega = \frac 12 (\lambda_- - \lambda_+)
\end {eqnarray}
and $c$ is a constant defined below in \S 2.2. \ In fact, in the semiclassical limit (or also for large distance between the two wells) the splitting $\omega$ is exponentially small, as $\hbar$ goes to zero. \ Furthermore, in \cite {Sacchetti2} it has been also seen that for $\sigma$ less than a critical value
\begin {eqnarray*}
\sigma_{threshold} = \frac 12 \left [ 3 + \sqrt {13} \right ]
\end {eqnarray*}
then a supercritical pitchfork bifurcation occurs; on the other hand, for $\sigma$ bigger than the critical value $\sigma_{threshold}$ a subcritical pitchfork bifurcation associated to the appearance on a couple of saddle node points occurs.

It is worth mentioning the fact that the main problem consists in proving the stability of the two-level approximation (which basically is a two-mode problem) with respect to the NLS equation (\ref {Eq1}). \ So far, the stability of the two-level approximation has been proved, in the semiclassical limit, only for times of the order of the beating period $T=2\pi \hbar/\omega$ \cite {S}, or for exponentially large times (that is of the order $e^T$) under further assumptions as proved by \cite {BaSa}. \ In fact, our previous approach was rather efficient in order to study the dynamics, but only give a partial result in order to look for the stationary solutions. \ Recently, Kirr, Kevrekidis, Shlizerman and Weinstein \cite {KKSW} has considered the stationary solution problem for the Cauchy problem (\ref {Eq1}) with $\hbar $ fixed (i.e. $\hbar =1$) in the limit of large barrier between the two wells, and in the case of cubic nonlinearities. \ In their seminal paper they make use of the Lyapunov-Schmidt reduction method to the two-level approximation equation for the stationary solutions. \ In such a way they overcome the limit of the method applied by \cite {S} for the study of the stationary solutions. \ Furthermore, they also applied the same method in order to study the orbital stability of the obtained solutions. 

In this paper we follow the ideas developed by \cite {KKSW}, adapted to the semiclassical limit and considering the case of any positive and real nonlinearity power $\sigma$, in order to study the stationary solutions of equation (\ref {Eq1}) and their stability properties as function of the nonlinearity power $\sigma$. \ In particular we are able to prove that the result obtained by \cite {Sacchetti2} for the two-level approximation, concerning the existence on the critical value $\sigma_{threshold}$, holds true for the whole Cauchy problem (\ref {Eq1}), too. \ To this end we prove the stability of the two-level approximation, when restricted to the stationary problem, and then we apply a generalization of the Budan-Fourier theorem \cite {CLLLR} in order to count all the branches associated to the stationary solutions. 

It is worth to mention the fact that the stability of the two-level approximation holds true in order to classify the stability/instability properties of the stationary solutions, too. \ In fact, stability/instability properties of the stationary solutions for the two level approximation are easily obtained since such an approximation has a finite-dimensinal Hamiltonian structure. \ On the other side, orbital stability/instability properties of the stationary solutions of the full nonlinear problem are much harder to obtain. \ However, in this paper, by making use of the methods developed by Grillakis, Shatah and Strauss \cite {G,GSS}, and succesfully applied by \cite {KKSW} for double well problems with cubic nonlinearity, we prove the equivalence between the stability/instability properties when we restrict our problem to the case of attractive nonlinearity and when we restrict our analysis to the "ground state".

There are already many studies on the existence of stationary solutions and the stability of Eq.(\ref{Eq1}) in the semiclassical limit (e.g., \cite{FlWe,G,GSS}). \ However, our aim is to understand what happens with double-well problem. \ When we consider the stationary problem with symmetric double-well and nonlinearity strength large enough, the bifurcation picture tells us that we have asymmetrical stationary solutions localized on just one well, as well as asymmetrical stationary solution delocalized between the two wells. \ The first type of solution was obtained, but the second type of solution was not considered in \cite{FlWe}, and it is identified with the multi-bump stationary solution studied in, e.g., \cite{DeFe}. \ Also it would be important to understand 
the destruction of the beating motion in the framework of the dynamics (see \cite{GMS} for related topics).

The paper is organized as follows. \ In Section \ref {Sec2} we recall some preliminary spectral results for Schr\"odinger operator with double well potential in the semiclassical limit, we introduce the main assumptions and we collect some general global well-posedness results for the Cauchy problem (\ref {Eq1}). \ In Section \ref {Sec3} we prove (Theorem \ref {Theorem1}) concerning the occurrence and the nature of spontaneous symmetry breaking phenomenon for equation (\ref {Eq1}) by applying, in the semiclassical limit, the Lyapunov-Schmidt reduction method to the two-level approximation and some results of the theory of numbers in order to count the number of solutions of a polynomial-type equation coming from the two-level approximation. \ In Section \ref {Sec4} we consider the dynamical properties of the stationary solutions of the two-level approximation, which has Hamiltonian form. \ In Section \ref {Sec5} we consider the orbital stability properties of the ground state stationary solutions. Appendix is devoted to an application of all the arguments in the previous sections to an explicit one dimensional toy model where the double well potential is given by a couple of attractive Dirac's delta interactions.

{\it Notations.} \ Hereafter, 

\begin {itemize}

\item [$\bullet$] $y = \tilde \asy (x )$, means that for any $0 < \alpha <1$ there exists a positive constant 
$C:= C_{\alpha }$ such that $|y| \le C_{\alpha } |x|$. \ Here, as usual $y = \asy (x )$ means that there exists a 
positive constant $C$ such that $|y| \le C |x| $, and $x\sim y$ means that $\lim_{\hbar \to 0 } \frac {x}{y} =C$ 
for some $C \in \R$;

\item [$\bullet$] $\| \cdot \|_p$ and $\| \cdot \|$ denote the norm of the spaces $L^p$ and $L^2$, $\langle \phi , \varphi \rangle = \int \bar \phi \varphi $ denotes the scalar product in the Hilbert space $L^2$;

\item [$\bullet$] $C$ denotes any positive constant which value is independent of $\hbar$.

\end {itemize}

\section {Main assumptions and preliminary results} \label {Sec2}

Here, we recall some preliminary results. \ Throughout the paper we always assume 
the Hypotheses below in this section.  

\subsection {Linear operator}

Here, we introduce the assumptions on the double-well potential $V$ and we collect some well known results on the linear operator $H_0$.

\begin {hypothesis} \label {Hyp1}
The potential $V (x)$ is a bounded real valued function such that:

\begin {itemize}

\item [{i.}] $V$ is a symmetric potential. \ For the sake of definiteness we can always assume that, by means of a suitable choice of the coordinates, $V$ is symmetric with respect to the spatial coordinate $x_1$, that is 
\be
[ {\mathcal S} , V ] =0 
\label {Eq6}
\ee
where 
\bee
\left [ {\mathcal S} \psi \right ] (x_1 , x_2, \ldots , x_d ) = \psi (-x_1 , x_2 , \ldots , x_d) .
\eee
Hence, the Hamiltonian $H_0$ is invariant under the space inversion: $[{\mathcal S},H_0]=0$, 

\item [{ii.}] $V \in C^\infty (\R^d )$;

\item [{iii.}] $V (x)$ admits two minima at $x=x_\pm $, where $x_- = {\mathcal S} x_+ \not= x_+$, such that 
\be
V (x) > V_{min} = V(x_\pm ) , \ \ \forall x \in \R^d , \ x\not= x_\pm  . \label {Eq7}
\ee
For the sake of simplicity, we assume also that 
\bee
\nabla V (x_\pm ) =0 \ \ \mbox { and } \ \ \mbox { Hess } V (x_\pm  ) >0 . 
\eee

\item [{iv.}] Finally we assume that the two minima are not degenerate:
\be
V^-_\infty = {\liminf}_{|x|\to \infty}  V(x)  >V_{min} \, . \label {Eq8}
\ee
\end {itemize}
\end {hypothesis}

\begin {remark}
In fact, some assumptions on $V$ may be weakened. \ In particular, the case of degenerate minima, that is $ \mbox {\rm det } [\mbox {Hess } V(x_\pm )]=0$, could be treated in a similar way; however, we don't dwell here on such details. \ Furthermore, boundedness of $V$ is assumed just for sake of definiteness: if $V$ is not bounded we could make use of the argument by \cite {BaSa} in order to prove the well-posedness of the Cauchy problem (\ref {Eq1}), under some assumptions of the behavior of the potential at infinity. \ For instance, we could assume that there exists a positive constant $0 < m \le 2$ such that for large $|x|$
\bee
C \langle x \rangle ^m \le V(x) \le C^{-1} \langle x\rangle ^m \, , \ \langle x\rangle  =  (1+|x|^2)^{1/2},  
%{1+\sum_{j=1}^d |x_j|}\, , 
\eee
for some $C>0$, and
\bee
\left | \partial_{x_1}^{\alpha_1} \ldots \partial_{x_d}^{\alpha_d} V (x) \right | \le C_\alpha \langle x \rangle^{m-|\alpha |}\, , \ |\alpha |= \sum_{j=1}^d \alpha_j \, ,
\eee
for any multi-index $\alpha \in \N^d$.
\end {remark} 

The operator $H_0$ formally defined by (\ref {Eq2}) admits a self-adjoint realization (still denoted by $H_0$) on $H^2 (\R^d )$ since $V$ is a bounded potential.

Let $\sigma (H_0 ) = \sigma_d \cup \sigma_{ess}$ be the spectrum of the self-adjoint operator $H_0$, where $\sigma_d$ denotes the discrete spectrum and $\sigma_{ess}$ denotes the essential spectrum. \ It follows that
\bee
\sigma_d \subset (V_{min},V_\infty^- ) 
\ \ \mbox { and } \ \ \sigma_{ess}=[V_\infty^- ,+\infty )  \, . 
\eee

Furthermore, for any $\hbar \in (0,\hbar^\star )$, for some $\hbar^\star >0$ fixed and small enough, it follows that $\sigma_d$ is not empty and, in particular, it contains two eigenvalues at least $\lambda^1_+$ and $\lambda^1_-$ where $\lambda^1_+ < \lambda^1_-$ and 
\be
\inf_{\zeta \in \sigma (H_0) \backslash \{ \lambda_{\pm}^1 \} } [\zeta - \lambda_{\pm}^1 ] \ge C \hbar \, ,  \label {Eq9}
\ee
for some positive constant $C$ independent of $\hbar$.

\begin {remark} \label {Remark2}
Actually, from Hypothesis \ref {Hyp1} and for $\hbar$ small enough in general it follows that for some $E>V_{min}$ then  
\bee
\sigma_d \cap  (V_{min},E)
\eee
is given by a sequence of couple of nondegenerate eigenvalues $\lambda_\pm^j$, $j=1,2,\ldots ,n$ where $n \sim \hbar^{-1}$, such that $\lambda^j_+ < \lambda_-^j$ and 
\be
\inf_{\zeta \in \sigma (H_0) \backslash \{ \lambda_{\pm}^j \} } \left | \zeta - \lambda_{\pm}^j \right | \ge C \hbar \,  \label {Eq10}
\ee
hold true. \ In fact, degeneracy may occur for some $j>1$ only in special cases, for instance when other symmetry properties for the potential $V$ are present (see, e.g., \cite {H}). \ Hereafter, for the sake of definiteness, we assume that degeneracy does not occur and that (\ref {Eq10}) holds true for any $j=1,2,\ldots ,n$.
\end {remark}

Let $\varphi_{\pm}^j$ be the normalized eigenvectors associated to $\lambda_{\pm}^j$, then $\varphi_{\pm}^j$ can be chosen to be real-valued functions such that 
\be
{\mathcal S} \varphi_{\pm}^j = \pm \varphi_{\pm}^j ; \label {Eq11}
\ee
Furthermore

\begin {lemma} \label {Lemma1}
The eigenvectors $\varphi_\pm^j$ belong to the space $H^2 (\R^d ) \cap L^p (\R^d)$ where 
\be
2 \le p \ 
\left \{
\begin {array}{ll}
\le +\infty  & \mbox { if } d=1 \\
< +\infty  & \mbox { if } d=2 \\
< 2d/(d-2) & \mbox { if } d>2 
\end {array}
\right. \, . \label {Eq12}
\ee
In particular, it follows that
\be
\| \nabla \varphi_\pm^j \| \le C_j \hbar^{-1/2} \ \mbox { and } \ \| \varphi_\pm^j \|_{H^2} \le C_j \hbar^{-1} \label {Eq13}
\ee
and
\be
\| \varphi_{\pm }^j \|_p  \le C_j \hbar^{- d\frac {p-2}{4 p}} \, ,  \label {Eq14}
\ee
for some positive constant $C_j$, independent on $\hbar$.
\end {lemma}

\begin {proof}
Indeed, $\varphi_\pm^j$ is normalized and it satisfies to the following eigenvalue equation $-\hbar^2 \Delta \varphi^j_\pm =(\lambda^j_\pm -V) \varphi^j_\pm$, from which immediately follows that 
\bee
\hbar^2 \| \nabla \varphi_\pm^j \|^2 &=& \langle (\lambda^j_\pm -V) \varphi^j_\pm, \varphi^j_\pm \rangle \\
&\le & \langle (\lambda^j_\pm -V) \varphi^j_\pm, \varphi^j_\pm \rangle_{L^2(\Omega^j_\pm)} \\ &\le & C_j \hbar \| \varphi^j_\pm \|^2 
\eee
where 
\bee
\Omega^j_\pm = \{ x \in \R^d \ | \ V(x) \le \lambda^j_\pm \}
\eee
is such that $\lambda^j_\pm -V \ge \lambda^j_\pm -V_{min} \ge C_j \hbar$ for any fixed $j$ and $\hbar$ small enough. \ Similarly
\bee
\hbar^2 \| \Delta \varphi_\pm^j \|^2 = \left \| (\lambda^j_\pm -V) \varphi^j_\pm \right \| \le C_j \| \varphi^j_\pm \| \, . 
\eee
since $V$ is a bounded potential. \ Estimate (\ref {Eq14}) follows by means of the Gagliardo-Nirenberg inequality:
\bee
\| \varphi_\pm^j \|_p \le C \| \nabla \varphi_\pm^j \|^\delta \| \varphi_\pm^j \|^{1-\delta} \le C \hbar^{-\delta/2}
\eee
where $\delta = \frac {p-2}{2p}d$.
\end {proof}

\begin {remark}
Actually, $\varphi_\pm^j \in L^p$ for any $p$ and, by means of the Riesz-Thorin interpolation Theorem, inequality (\ref {Eq14}) holds true for any $p$ independently on the dimension $d$ (see, e.g., \cite {S}). \ Indeed, by means of the semiclassical expression of $\varphi_j$ it follows that $\| \varphi_j \|_\infty \le C_j \hbar^{-d/4}$.
\end {remark}

The \emph {splitting} between the two eigenvalues 
\be
\omega^j = \case 12 (\lambda_{-}^j -\lambda_{+}^j ) \label {Eq15}
\ee
vanishes as $\hbar $ goes to zero. \ In order to give a precise estimate of the splitting $\omega^j$ we make use of the fact that $V$ is a symmetric double-well potential with non-zero barrier between the wells. \ That is, let $j$ be fixed and let
\be
\rho = \inf_\gamma \int_{\gamma} \sqrt {\left [ V (x)-V_{min} \right ]_+ }\, d x >0, \label {Eq16}
\ee
be the Agmon distance between the two wells; where $\gamma$ is any path connecting the two wells, that is $\gamma \in AC ([0,1],\R^d)$ such that $\gamma (0)=x_-$ and $\gamma (1)=x_+$, and where $[\cdot ]_+ = \max (\cdot , 0 )$. \ From standard WKB arguments (see \cite{H} for details) then it follows that the splitting is \emph {exponentially small}, that is 
\be
\omega^j = \tilde \asy ( e^{- \rho /\hbar } ) \, . \label {Eq17}
\ee

Let $\varphi_{R,L}^j$ be the normalized \emph {single well states} associated to the linear eigenstates  $\varphi^j_\pm$ by means of 
\begin{eqnarray}
\varphi_R^j ={(\varphi_+^j + \varphi_-^j )}/{\sqrt {2}} \label {Eq18}
\end {eqnarray}
and 
\begin {eqnarray} 
\varphi_L^j =  {(\varphi_+^j - \varphi_-^j )}/{\sqrt {2}},\label {Eq19}
\end{eqnarray}
They are \emph {localized on one well} in the sense that and for any $p \in [2,+\infty ]$ then 
\be
\| \varphi_{R}^j \varphi_{L}^j \|_p = \tilde \asy (  e^{-\rho /\hbar })\,  . 
\label {Eq20}
\ee
More precisely, these functions are localized on only one of the two wells in the sense that for any $r>0$ there exists $c:=c(r)>0$ such that 
\bee
\int_{D_r(x_+)} |\varphi_{R}^j (x)|^2 dx = 1 + \asy (e^{-c /\hbar }) 
\eee
and
\bee 
\int_{D_r(x_-)} |\varphi_{L}^j (x)|^2 dx = 1 + \asy (e^{-c /\hbar })  
\eee
where $D_r(x_\pm )$ is the ball with center $x_\pm$ and radius $r$. \ For such a reason we call them \emph {single-well} (normalized) states. 

\begin {remark} \label {Remark4}
In the following, for the sake of definiteness, we restrict ourselves to the couple of eigenvalues $\lambda_+^1$ and $\lambda_-^1$, corresponding to the lowest energies. \ Hereafter, we simply denote them by $\lambda_\pm$ dropping out the index $1$, and $\varphi_\pm$ denote the associated eigenvectors. \ The symmetric solution $\varphi_+$ is the first eigenfunction of $H_0$, so it is positive.  \ We remark that the existence of the stationary solutions for (\ref {Eq1}) and their dynamical stability still hold true when we consider all the unperturbed energy levels $\lambda_\pm^j$ provided that degeneracy does not occur as discussed in Remark \ref {Remark2}. 
\end {remark}

\subsection {Assumption on the non linear term}

In order to obtain some a priori estimates of the wavefunction $|\psi |^{2\sigma} \psi$ we introduce the following assumption on the nonlinearity power $\sigma$.

\begin {hypothesis} We assume that
\be
0 < \sigma < 
\left \{
\begin {array}{ll}
+ \infty & \ \mbox { if } \ d=1,\ 2 , \\
\frac {1}{d-2} & \ \mbox { if } \ d >2
\end {array}
\right. \, . \label {Eq21}
\ee
where $d$ is the spatial dimension.
\end {hypothesis}

Let
\bee
C_R = \langle \varphi_R^{\sigma +1} , g \varphi_R^{\sigma+1} \rangle  \ \mbox { and } \ C_L = \langle \varphi_L^{\sigma +1} , g \varphi_L^{\sigma+1} \rangle 
\eee
where $C_R =C_L$ because of the symmetric properties of $g$ and $V$. \ We assume also the following scaling limit.

\begin {hypothesis} \label {Hyp3} Let $\omega = \frac 12 (\lambda_- - \lambda_+)$ be the splitting (\ref {Eq15}) satisfying to the asymptotic estimate (\ref {Eq17}). \ We assume that the real-valued parameter $\epsilon$ depends on $\hbar$ in such a way
\be
|\eta |\le C \  \mbox { where } \ \eta = \frac {\epsilon c}{\omega} \, , \ \  c:= C_R = C_L ,  \label {Eq22}
\ee
for some positive constant $C$, independent of $\hbar$. \ The parameter $\eta$ plays the role of effective nonlinearity parameter. \ Hereafter, we assume that $g(x)$ has the same symmetry property (\ref {Eq6}) of the potential $V$ and it is such that $\langle \varphi_R^{\sigma +1}, g \varphi_R^{\sigma +1} \rangle \not= 0 $. \ In particular, for the sake of definiteness, let 
\be
\langle \varphi_R^{\sigma +1}, g \varphi_R^{\sigma +1} \rangle > 0 \, . \label {Eq23}
\ee
\end {hypothesis}

\subsection {Existence results in $H^1$ and conservation laws.} \label {Sec2c}

The results below follow from \cite {C} and from the a priori estimate given by \cite {S}.

\subsubsection {Local existence in $H^1$} 
Let the initial state $\psi^0 \in H^1 $, then there exists $T^\star >0$ and an unique solution $\psi (x,t) \in C([0,T^\star ),H^1) \cap C^1 ([0,T^\star), H^{-1})$ of (\ref {Eq1}), where $T^\star =+\infty $ or $\| \nabla \psi \| \to +\infty$ as $t\to T^\star -0$. 
\ Furthermore, the conservation of the norm and of the energy hold true 
for $t\in [0,T^\star]$:
\bee
\| \psi (\cdot , t) \| = \| \psi^0 (\cdot )\| 
\eee
and
\bee
\tilde {\mathcal H} (\psi (\cdot , t )) = \tilde {\mathcal H} (\psi^0 (\cdot ))
\eee
where
\bee
\tilde {\mathcal H} (\psi ) = \langle \psi , H_0 \psi \rangle + \frac {\epsilon}{\sigma +1} \langle \psi^{\sigma +1} , g \psi^{\sigma +1} \rangle 
\eee
represents the energy functional.

\subsubsection {Global existence}

The solution $\psi$ of (\ref {Eq1}) globally exists, that is $T^\star =+\infty$, provided that the state is initially prepared on the first $N$ states of the linear problem, for any $N$ fixed, and $\hbar$ is small enough. \ Indeed, this fact immediately follows from a priori estimate of the norm of the gradient of the wavefunction \cite {S}.
\begin {remark}
The solution $\psi (x,t)$ globally exists for both positive and negative values of the parameter $\epsilon$, provided that $\hbar$ is small enough and $\epsilon$ satisfies Hyp. \ref {Hyp3}. \ That is, because of the scaling assumptions, blow-up effect cannot occur.
\end {remark}

\section {Stationary solutions and bifurcation} \label {Sec3}

Since the beating period $T= \frac {2\pi \hbar}{\omega}$ plays the role of the unit of time it is convenient to introduce the \emph {slow time}
\bee
\tau = \frac {\omega t}{\hbar } \, , 
\eee
then equation (\ref {Eq1}) takes the form (here $'$ denotes the derivative with respect to $\tau $ and where, with abuse of notation, $\psi = \psi (\tau , x)$)
\be
i \omega \psi' = H_0 \psi + \epsilon g |\psi |^{2\sigma } \psi \, , \ \| \psi (\cdot , \tau ) \| =1 . 
\label {Eq24}
\ee

In order to study the stationary solution we set 
\bee
\psi (x,\tau ) = e^{-i \lambda \tau /\omega}  \psi (x), \quad \| \psi (\cdot ) \| =1 \, , \  \lambda = {\Omega}+{\omega}  E , 
\eee
where
\bee
\Omega = \frac 12 \left [ \lambda_+ + \lambda_- \right ] \, . 
\eee
As specified in Remark \ref {Remark4} we restrict ourselves, for the sake of definiteness, to the first couple of energy level $\lambda^1_\pm$, where we simply denote them by $\lambda_\pm$ dropping out the index $1$; similarly $\varphi_\pm$ denote the associated eigenvectors and $\varphi_{R,L}$ the associated single-well states.

Hence, equation (\ref {Eq24}) takes the form 
\be
\lambda \psi = H_0 \psi + \epsilon g |\psi |^{2\sigma } \psi, \quad \| \psi (\cdot ) \| =1 . 
\label {Eq25}
\ee
Now, let us set 
\be
\psi (x) = a_R \varphi_R (x)+ a_L \varphi_L (x) + \psi_c (x) \, , \label {Eq26}
\ee
where
\bee
\psi_c (x ) = \Pi_c \psi (x ) 
\eee
and
\bee
a_R = \langle \varphi_R , \psi \rangle \ \ \mbox { and } \ \ a_L = \langle \varphi_L , \psi \rangle 
\eee
are unknown complex-valued values. \ Here,
\bee
\Pi_c = \I - \Pi , \ \Pi = \left [ \langle \varphi_+ , \cdot \rangle \varphi_+ + \langle \varphi_- , \cdot \rangle \varphi_- \right ] 
\eee
denotes the projection operator onto the eigenspace orthogonal to the bi-dimensional space associated to the doublet $\{ \lambda_\pm \}$. 

Since
\be
H_0 \psi &=& a_R H_0 \varphi_R + a_L H_0 \varphi_L + H_0 \psi_c \nonumber \\ 
&=& a_R \left [ \Omega \varphi_R - \omega \varphi_L \right ] + a_L \left [ -\omega \varphi_R + \Omega \varphi_L \right ] + H_0 \psi_c \label {Eq27}
\ee
then, by substituting (\ref {Eq26}) in (\ref {Eq25}) and projecting the resulting equation onto the one-dimensional spaces spanned by the \emph {single-well} states $\varphi_R$ and $\varphi_L$, and on the space $\Pi_c L^2 (\R^d)$ it follows that (\ref {Eq25}) takes the form
\be
\left \{
\begin {array}{ll}
E  a_R = - a_L + r_R, &  r_R = r_R (a_R , a_L , \psi_c ) = \frac {\epsilon}{\omega} \langle \varphi_R , g |\psi |^{2 \sigma}  \psi \rangle \\ 
E  a_L = - a_R + r_L, & r_L = r_L (a_R , a_L , \psi_c ) = \frac {\epsilon}{\omega} \langle \varphi_L , g |\psi |^{2 \sigma} \psi \rangle \\ 
E \psi_c = \frac {1}{\omega} \left [ H_0-\Omega \right ] \psi_c + r_c, & r_c = r_c (a_R , a_L , \psi_c ) = \frac {\epsilon}{\omega} \Pi_c g |\psi |^{2 \sigma} \psi 
\end {array}
\right. \label {Eq28}
\ee
with the normalization condition
\bee
|a_R|^2 + |a_L|^2 + \| \psi_c \|^2 = 1 \, . 
\eee

\begin {remark}
Since equations (\ref {Eq28}) has stationary solutions (\ref {Eq26}) defined up to a phase term then we can always assume, for the sake of definiteness, that the stationary solution of equation (\ref {Eq25}) is such that $a_L$ is a real-valued positive constant: $a_R\in \C$ and $a_L \in \R^+$. \ Furthermore, we remark that $[H_0 , {\mathcal S}]=0$ and $[g,{\mathcal S}]=0$; hence, if $\psi$ is a stationary solution of equation (\ref {Eq25}) associated to a given value $\lambda$, then ${\mathcal S}\psi$ is a solution associated to the same level, too. 
\end {remark}

Then, collecting the results from Lemmata 3 and 4 (and the associated remarks) by \cite {S} we have the following.

\begin {lemma} \label {Lemma2}
Let $\rho$ be the Agmon distance between the two wells defined as in (\ref {Eq16}). \ It follows that 
\bee
r_{R,L} (a_R, a_L, \psi_c) = r_{R,L} (a_R, a_L, 0) + r_{R,L}^c (a_R,a_L,\psi_c) 
\eee
where

\begin {itemize}

\item [(i)] 
\be
r_{R,L} (a_R, a_L, 0)= \frac {\epsilon}{\omega} C_{R,L} |a_{R,L}|^{2\sigma} a_{R,L} + \tilde \asy (e^{-\rho /\hbar } ) \label {Eq29}
\ee
and
\be
C_{R,L} &=&  \langle \varphi_{R,L} , g |\varphi_{R,L} |^{2\sigma } \varphi_{R,L} \rangle = \langle \varphi_{R,L}^{\sigma +1}, g \varphi_{R,L}^{\sigma +1} \rangle = \asy \left ( \hbar^{-d\sigma /2} \right ); \label {Eq30}
\ee
by the symmetry assumptions it turn out that 
\bee 
C_R &=& C_L \, . 
\eee

\item [(ii)] The remainder terms are estimated as follow
\bee
|r_{R,L}^c | \le \frac {\epsilon}{\omega} C \hbar^{-d \sigma /2} \| \psi_c \|^\gamma  
\eee
where
\be
\gamma = 
\left \{
\begin {array}{ll}
1 & \ \mbox { if } \ d=1,2 \\ 
1+ (2-d)\gamma & \ \mbox { if } \ d >2
\end {array}
\right. \, . \label {Eq31}
\ee
\end {itemize}

\end {lemma}

Here we come with the existence result of stationary states for the nonlinear Schr\"odinger equation (\ref {Eq25}).

\begin {theorem} \label {Theorem1}
Let 
\be
a_R = p e^{i\theta }\, , \ a_L = q \ \mbox { and }\ z = p^2 - q^2
\label {Eq32}
\ee
where $p,q \in [0,1]$ and $ \theta \in [0, 2 \pi)$. \ Let $\hbar \in (0, \hbar^\star )$, where $\hbar^\star $ is small enough, let $\rho$ be the Agmon distance between the two wells and let $\eta$ be the effective nonlinearity defined by (\ref {Eq22}). \ Then the stationary problem (\ref {Eq28}) always has 

\begin {itemize}

\item [-] a {\bf symmetric} solution $\psi_E^s$ such that 
\bee
\theta^s = \tilde \asy (e^{-\rho \gamma /\hbar }), \ z^s =\tilde \asy (e^{-\rho \gamma /\hbar }),
\eee
associated to
\bee 
E:=  -1 + \eta \frac {1}{2^\sigma} +\tilde \asy (e^{-\rho  \gamma/\hbar }),
\eee

\item [-] an {\bf antisymmetric} solution $\psi_E^a$ such that 
\bee
\theta^a = \pi + \tilde \asy (e^{-\rho  \gamma/\hbar }), \ z^a=\tilde \asy (e^{-\rho  \gamma/\hbar }), 
\eee 
associated to
\bee
E := +1 + \eta \frac {1}{2^\sigma} + \tilde \asy (e^{-\rho  \gamma/\hbar }).
\eee

\end {itemize}

Furthermore, in the case of negative (resp. positive) $\eta $, then  asymmetrical solution $\psi_E^{as}$ corresponding to $\theta^{as} =\tilde \asy (e^{-\rho  \gamma/\hbar })$ (resp. $\theta^{as} =\pi + \tilde \asy (e^{-\rho \gamma /\hbar })$) may appear as a result of spontaneous symmetry bifurcation phenomenon. \ That is: 
\begin {itemize}

\item [-] for $\sigma \le \sigma_{threshold}$ the symmetric (resp. antisymmetric) state corresponding to $z^s =\tilde \asy (e^{-\rho  \gamma/\hbar })$ bifurcates showing a pitchfork bifurcation when the adimensional nonlinear parameter $|\eta |$ is larger than the critical value $\eta^\star$ given by (see Fig. \ref {Fig1}, panel (a))
\bee
\eta^\star = \frac {2^\sigma}{\sigma}
\eee

\item [-] for $\sigma > \sigma_{threshold}$ two couples of new {\bf  asymmetrical} stationary states appear as saddle-node bifurcations when $|\eta |$ is equal to a given value $\eta^+$ such that $\eta^+ < \eta^\star$; then, for increasing values of $|\eta |$ two branches of the solutions disappear at $|\eta | = \eta^\star$ showing a subcritical pitchfork bifurcation (see Fig. \ref {Fig1}, panel (b)). \ The critical value $\eta^+$ is given by $\eta (z^+)$ where 
\be
\eta (z) = \frac {2z}{\sqrt {1-z^2}} \left [ \left ( \frac {1+z}{2} \right )^{\sigma} - \left ( \frac {1-z}{2} \right )^{\sigma} \right ]^{-1} \label {Eq33}
\ee
and $z^+ \in (0,1)$ is the non zero solution of the equation $\eta' (z)=0$.

\end {itemize}

In all the cases, the remainder term $\psi_c$ of the stationary solutions is such that
\be
\| \psi _c \|_{H^2} = \tilde \asy (e^{-\rho /\hbar } )\, . \label {Eq34}
\ee

The critical value $\sigma_{threshold}$ is given by
\bee
\sigma_{threshold} = \frac 12 \left [ 3 + \sqrt {13} \right ]
\eee
and it is an universal value in the sense that it does not depend on the shape of the double well potential as well as on the dimension $d$.
\end {theorem}

\begin {remark} \label {Remark7}
Concerning the symmetric solution $\psi_E^s = a_R^s \varphi_R + a_L^s \varphi_L + \psi_c$ we should remark the above statement implies that the corresponding level $E$ is non-degenerate in the sense that we have only this stationary solution corresponding to such value of $E$. \ On the other side, by means of a symmetrical argument, then ${\mathcal S} \psi_E^s = a_R^s \varphi_L + a_L^s \varphi_R + {\mathcal S} \psi_c$ is a solution associated to same level $E$, too. \ Hence, $\psi_E^s$ and ${\mathcal S} \psi_E^s$ coincide, up to a phase factor. \ From this fact and from Theorem \ref {Theorem1} it turns out that $\theta^s$ and $z^s$ are {\bf exactly} zero:
\bee
\theta^s = 0 \ \mbox { and } \ z^s =0 \, .
\eee
Similarly, it follows that 
\bee
\theta^a = \pi \ \mbox { and } \ z^a =0 \, 
\eee
and 
\bee
\theta^{as} = 0 \ (\mbox {respectively } \theta^{as} = \pi )
\eee
for negative value of $\eta$ (resp. for positive value of $\eta$). \ By means of a similar argument we can also conclude that the stationary solution is, up to a phase term, a real valued function; indeed if $\psi$ is a solution associated to a given level $E$, then $\bar \psi$ is a solution associated to the same value $E$, too.
\end {remark}

\begin{remark}
Because of the technical assumptions on $\sigma$, this critical value $\sigma_{threshold}$ makes sense for the non-linear Schr\"odinger equation (\ref {Eq25}) only in dimensions 1 and 2. \ This is not the case when we restrict our analysis to the two-level approximation.
\end{remark}

\begin {remark}
From Theorem \ref {Theorem1} it appears that we have only two pictures, accordingly with the value of $\sigma$. \ In Fig. \ref {Fig1} (panel (a)) we consider the bifurcation scenario for the imbalance function $z=|a_R|^2-|a_L|^2$ appearing when $\sigma \le \sigma_{threshold}$. \ In Fig. \ref {Fig1} (panel (b)) we consider the bifurcation scenario appearing when $\sigma > \sigma_{threshold}$. \ The same picture has been previously obtained for the two-level approximation (see, e.g., \cite {Sacchetti2}) where we have taken $\psi_c =0$; in fact, $\psi_c$ is exponentially small as proved in Theorem \ref {Theorem1}.
\end {remark}

\begin{remark}
The stationary solutions $\psi := \psi_E$, associated to the level $E$, given in Theorem \ref {Theorem1} are such that 
\be
\| \nabla \psi_E \| \le C \sqrt {\Lambda} \label {Eq35}
\ee
and
\be
\| \psi_E \|_p \le C \Lambda^{d \frac {p-2}{4p}} \label {Eq36}
\ee
where $p$ satisfies condition (\ref {Eq12}) and where 
\bee
\Lambda = \frac {{\mathcal H} (\psi_E) - V_{min}}{\hbar^2} \sim \hbar^{-1}
\eee
and  
\bee
{\mathcal H} (\psi ) = \langle \psi , H_0 \psi \rangle + \frac {\epsilon}{\sigma +1} \langle \psi^{\sigma +1} , g \psi^{\sigma +1} \rangle 
\eee
is the energy functional defined on $H^1 (\R^d ) \cap L^{2(\sigma +1)} (\R^d )$. \ Indeed, estimates (\ref {Eq35}) and (\ref {Eq36}) hold true for any vector $\psi$ belonging to the space $\Pi (L^2)$ (see Theorem 2 in \cite {S}). \ The results finally follow from this fact and since $\Pi_c \psi_E = \tilde \asy (e^{-\rho/\hbar })$.
\end{remark}

\begin{figure}
\begin{center}
\includegraphics[height=5cm,width=5cm]{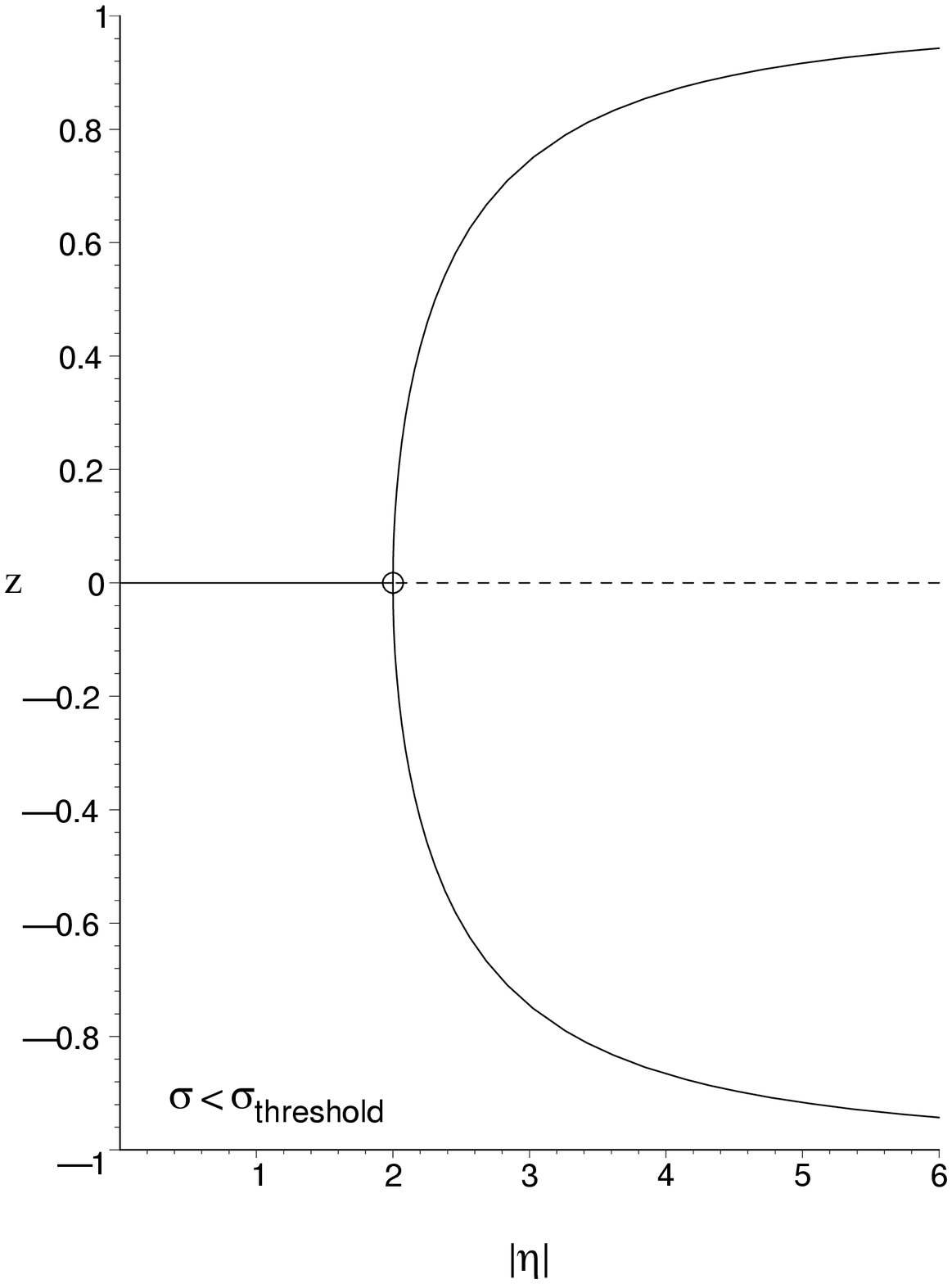}
\includegraphics[height=5cm,width=5cm]{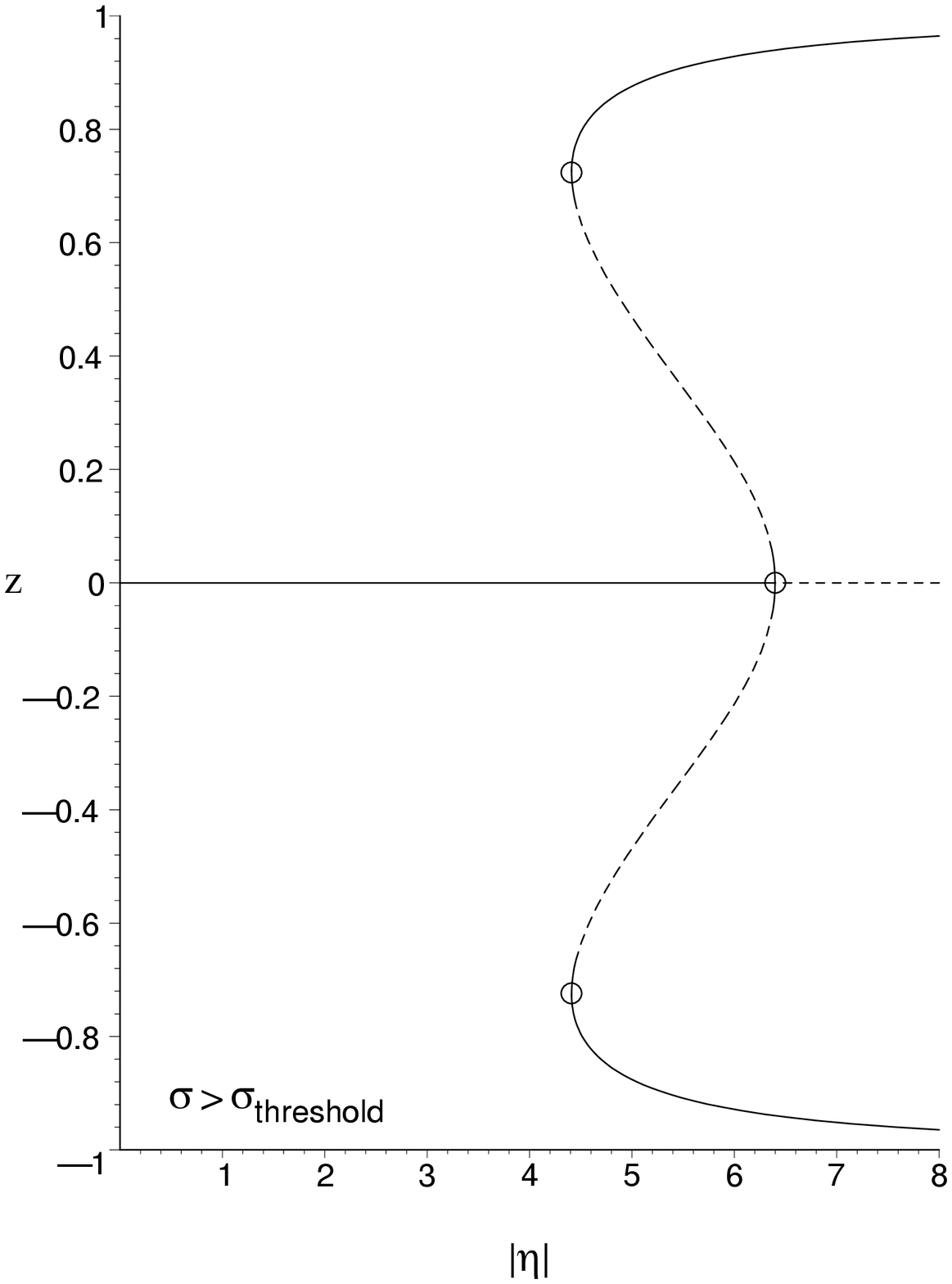}
\end{center}
\caption{\label {Fig1} In this figure we plot the graph of the stationary states of the non-linear Schr\"odinger equation (\ref {Eq25}) as function of the nonlinearity parameter $\eta$ for nonlinearity $\sigma =1 < \sigma_{threshold}$ (panel (a)) and for nonlinearity $\sigma =5 > \sigma_{threshold}$ (panel (b)); here $z = |a_R|^2 - |a_L|^2$ is the imbalance function. \ Full lines represent stable stationary states and broken lines represent unstable stationary states, where the notion of stability is referred to the dynamical stability associated to the Hamiltonian system given by the two-level approximation, as discussed in \S \ref {Sec4}; and also to orbital stability, as discussed in \S \ref {Sec5} in the case of attractive nonlinear case (i.e. $\eta <0$).}
\end{figure}

\subsection {Proof of Theorem \ref {Theorem1}}

Here, we prove the existence of the stationary solutions by making use of the Lyapunov-Schmidt method and applying some results of the theory of numbers in order to count the number of stationary solutions of the equation coming from the two-level approximation. \ In this section, for argument's sake, we take $\eta >0$; however, the same results still hold true also for $\eta <0$. 

\begin {lemma} \label {Lemma3}
We consider the following equation
\be
\left [H_0 -\Omega - \omega E \right ] \psi_c + \epsilon \Pi_c g |\psi |^{2\sigma } \psi =0 \, . 
\label {Eq37}
\ee
where the nonlinearity power $\sigma$ satisfies condition (\ref {Eq21}). \ For any fixed $C>0$ let 
\bee
D = \left \{ (a_R , a_L, E) \in \C^2 \times \R \ :\ |a_R|^2 + |a_L|^2 \le 1 , \ |\omega E |\le C \hbar^2  \right \} \, .
\eee
Then, for there exists $\hbar^\star >0$ small enough such that for any $\hbar \in (0, \hbar^\star )$ then there exists an unique solution $\psi_c \in H^2$ 
of equation (\ref {Eq37}) depending on $a_R$, $a_L$ and $E$, and such that 
\be
\max_{(a_R,a_L,E) \in D} \| \psi_c \|_{H^2} = \tilde \asy \left ( e^{-\rho /\hbar } \right )\, ,\ \mbox { as } \ \hbar \to 0 \, . \label {Eq38}
\ee
\end {lemma}

\begin {proof}
Recalling that
\bee
\psi = \varphi + \psi_c \, , \ \mbox { where}, \ \varphi = a_R \varphi_R + a_L \varphi_L \, , 
\eee
then (\ref {Eq37}) takes the form
\be
\psi_c = F (\psi_c ) \label {Eq38Bis}
\ee
where 
\be
F(\psi_c ):= F (\psi_c ; a_R, a_L, E)=- \epsilon \left [ H_0 - \Omega - \omega E \right ]^{-1} \Pi_c g |\psi |^{2\sigma } \psi \label {Eq38Ter}
\ee
and where
\be
\left \| \left [ H_0 - \Omega - \omega E \right ]^{-1} \Pi_c \right \|_{\mathcal{L}(L^2 \to H^2)} \le C_1 \hbar^{-1} \label {Eq39}
\ee
for some positive constant $C_1$ and for $\hbar$ small enough, since (\ref {Eq9}) and since $\omega E = \asy (\hbar^2 )$. \ On the other side we have that
\bee
\left \| F(u)-F(v) \right \|_{H^2} 
&\le & \epsilon \frac {C_2}{\hbar} \| |f|^{2\sigma} f - |g|^{2\sigma} g \| \\
&\le & \epsilon \frac {C_2}{\hbar} \|(|f|^{2\sigma} + |g|^{2\sigma})|f-g|\| \\
&\le & \epsilon \frac {C_2}{\hbar} (\|f\|_{H^1}^{2\sigma} +\|g\|_{H^1}^{2\sigma}) 
\|f-g\|_{H^1}
\eee
for some positive constant $C_2$, where we set 
\bee
f = \varphi + u~ \mbox { and }~ g = \varphi + v, 
\eee
with $|a_R|^2 + |a_L|^2 +\|u\|^2=1$, $|a_R|^2 + |a_L|^2 +\|v\|^2=1$. \ We have indeed made use of the H\"older inequality and of the Gagliardo-Nirenberg inequality  with $\sigma$ satisfying condition (\ref {Eq21}): if $2p \sigma < b$ and $ 2p/(p-2)<b$ where $b=+\infty $ if $d=1,2$ and $b = 2d/(d-2)$ if $d>2$, i.e. $\sigma$ satisfies (\ref {Eq21}). \ Finally, we get the wanted estimate
\be
\left \| F(u)-F(v) \right \|_{H^2} 
\le \epsilon \frac {2^{2\sigma} C_2}{\hbar}  \left \{ \max \left [ \| \varphi + u \|_{H^2} , \| \varphi + v \|_{H^2} \right ] \right \}^{2\sigma } \, \| u-v \|_{H^2}  \label {Eq40}
\ee
provided that $\sigma$ satisfies condition (\ref {Eq21}).

Now, let $C_3 = \max [C_1, 2^{2\sigma} C_2 ]$ and let 
\bee
K = \left \{ u \in H^2 \ : \ \| u \|_{H^2} \le c(\hbar ) \right \}, \ \ c(\hbar ) 
= \max \left \{ \left [ \frac {\hbar}{2^{2\sigma+2} 3 C_3 \epsilon } \right ]^{1/2\sigma}, 
\| \varphi \|_{H_2} \right \} \, .
\eee
Since $\| \varphi \|_{H_2} = \asy (\hbar^{-1})$, by Lemma \ref {Lemma1}, and $\epsilon = \tilde O (e^{-\rho /\hbar })$ 
then $c(\hbar ) = \left [ \frac {\hbar}{2^{2\sigma+2} 3 C_3 \epsilon } \right ]^{1/2\sigma }$.

Then $F$ is an operator from $K$ to $K$; indeed, from (\ref {Eq39}) and (\ref {Eq40}) it follows that
\bee
\| F(u) \|_{H^2} \le 
\epsilon C_3 \hbar^{-1} \| u +\varphi \|_{H_2}^{2\sigma +1} 
\le \left [ 2 \epsilon C_3 \hbar^{-1} (2c)^{2\sigma} \right ] c(\hbar) 
=\frac{1}{2} c(\hbar) <c(\hbar)\, .
\eee
Moreover, $F(u)$ is a contraction in $K$:
\bee
\| F(u) - F(v) \|_{H^2} \le C_3 \epsilon \hbar^{-1} \left [ 2 c (\hbar ) \right ]^{2\sigma } \| u-v \|_{H^2} < \frac{1}{4} \| u-v \|_{H^2} \, .
\eee
Hence, equation 
\bee
F(u) =u 
\eee
admits a unique solution $\psi_c$ in $K$ for any $(a_R, a_L, E) \in D$ and any $\epsilon$ satisfying Hyp. \ref {Hyp3}. \ This solution is given by the limit of the following sequence $\{ u_n \}_{n=0}^\infty $ where
\bee
u_0 =0 \ \mbox { and } \ u_{n+1} = F (u_n) \, .
\eee
In particular (the convergence is in $H^2$)
\bee
\psi_c = \lim_{n\to + \infty} u_n = \sum_{j=1}^{+\infty} \left [ u_{j+1} - u_j \right ] = \sum_{j=1}^{+\infty} \left [ F(u_{j}) - F(u_{j-1}) \right ] 
\eee
Since
\bee
\left \| F(u_{j+1}) - F(u_{j}) \right \|_{H^2} 
&\le & C_3 \epsilon \hbar^{-1} [2 c(\hbar )]^{2\sigma} \left \| F(u_{j}) - F(u_{j-1}) \right \|_{H^2} \\ 
&\le & \left [ C_3 \epsilon \hbar^{-1} [2c(\hbar )]^{2\sigma} \right ]^{j+1} 
\left \| F(u_{0})  \right \|_{H^2}
\eee
then we have that 
\be
\| \psi_c \|_{H^2} &\le & \frac {1}{1-C_2 \epsilon \hbar^{-1} [2c(\hbar )]^{2\sigma}} \left \| F(u_{0})  \right \|_{H^2} \nonumber \\ 
&\le & \frac {1}{1-C_3 \epsilon \hbar^{-1} [2c(\hbar )]^{2\sigma}} C_3 \epsilon \hbar^{-1} \left \| 
a_R \varphi_R + a_L \varphi_L \right \|_{H^2}^{2\sigma +1} \nonumber \\
& =& \tilde \asy (e^{-\rho /\hbar} ) \label {Eq41}
\ee
Since the constants $C_1$ and $C_2$ depend on $a_R$, $a_L$ and $E$ in such a way that 
\bee
\max_{|\omega E |\le C \hbar^2  } C_1 < + \infty 
\eee
and
\bee
\max_{|a_R|^2 + |a_L|^2 \le 1} C_2 < + \infty 
\eee
then the estimate (\ref {Eq41}) uniformly holds true on the set $D$.
\end {proof}

\begin {remark} \label {Remark11}
By means of the same arguments it follows that $\psi_c \in H^2$, as function on $a_R$, $a_L$ and $E$, admits the first derivatives and in particular these derivatives satisfy estimate (\ref {Eq38}) in the sense that  
\be
\max_{(a_R,a_L,E) \in D} \left [ \left \| \frac {\partial \psi_c }{\partial E} \right \|_{H^2} ,\  \left \| \frac {\partial \psi_c }{\partial a_R} \right \|_{H^2},\  \left \| \frac {\partial \psi_c }{\partial a_L} \right \|_{H^2}\right ] = \tilde \asy \left ( e^{-\rho /\hbar } \right )\, ,\ \mbox { as } \ \hbar \to 0 \, . \label {Eq42}
\ee
We can also give an estimate of the dependence of $\psi_c$ on the parameter $\epsilon$; this estimate will be given in Lemma \ref {Lemma7}.
\end {remark}

Now, setting $\psi_c = \psi_c (a_R, a_L, E)$ in (\ref {Eq28}), let any $0<\rho' < \rho$ fixed, let
\be
\nu = e^{-\rho' \gamma /\hbar }   \label {Eq43} 
\ee
where $\gamma$ is defined in equation (\ref {Eq31}), and making use of Lemma \ref {Lemma2}, then (\ref {Eq28}) takes the form 
\be
\left \{
\begin {array}{lcll} 
E a_R &=& - a_L + \eta  |a_R|^{2\sigma } a_R & + \nu f_R (a_R , a_L, E)  \\ 
E a_L &=& - a_R + \eta  |a_L|^{2\sigma } a_L & + \nu f_L (a_R , a_L, E)  \\ 
1 &=& |a_R|^2 + |a_L|^2 & + \nu f_c (a_R , a_L, E)
\end {array}
\right. \, 
\label {Eq44}
\ee
where $f_R$, $f_L$ and $f_c$ are uniformly bounded on $D$ with their first derivatives. \ Since Lemma \ref {Lemma3} and Remark \ref {Remark11}, and recalling that $\epsilon /\omega = \eta / \langle \varphi_R^{\sigma +1} , g \varphi_R^{\sigma +1} \rangle = O (\hbar^{-d\sigma/2})$. 
\ From (\ref {Eq32}) then (\ref {Eq44}) takes the form
\bee
\left \{ 
\begin {array}{lcll}
E p &=& - q e^{-i \theta } + \eta p^{2\sigma +1} &+ \nu e^{-i \theta } f_R \\
E q &=& - p e^{i \theta }  + \eta q^{2\sigma +1} &+ \nu f_L \\
1   &=& p^2 +q^2 &+ \nu f_c 
\end {array}
\right.
\eee
By taking the real and imaginary part of the previous equations we obtain the following system 
\be
G (p,q,E,\theta ;\nu )=0 \label {Eq45}
\ee
on 
\bee
D' = \left \{ (p, q, E, \theta) \in [0,1]^2 \times \R \times [0,2\pi ) \ :\ p^2 + q^2 \le 1 , \ |\omega E | \le C \hbar^2  \right \} 
\eee
and where $G =(G_1,G_2,G_3,G_4)$ are given by 
\bee
G_1 &=& E- \frac {1}{1-\nu f_c} \left [ - 2 pq \cos \theta + \eta (p^{2\sigma +2} + q^{2\sigma +2} ) + \nu \Re (p e^{-i\theta } f_R + q f_L ) \right ] \\ 
&=& E + 2 pq \cos \theta - \eta (p^{2\sigma +2} + q^{2\sigma +2} ) + \nu f_1 \\
G_2 &=& (p^2+q^2) \sin \theta + \nu \Im (e^{-i\theta } f_R - p f_L ) = (p^2+q^2) \sin \theta + \nu f_2 \\ 
G_3 &=& (p^2 - q^2) \cos \theta + \eta p q (p^{2\sigma} - q^{2\sigma } ) + \nu \Re (q e^{-i\theta } f_R - p f_L ) \\ 
&=& (p^2 - q^2) \cos \theta + \eta p q (p^{2\sigma} - q^{2\sigma } ) + \nu f_3 \\ 
G_4 &=& p^2 + q^2 + \nu f_c -1 = p^2 + q^2 - 1 +\nu f_4 
\eee
where $f_j$, $j=1,2,3,4$, are uniformly bounded on the set $D'$ with their first derivatives.

From equations $G_2=0$ and $G_4 =0$ we obtain that 
\bee
p^2+q^2 =1+ \asy (\nu ) \ \mbox { and } \ \theta =\asy (\nu )\, , \ \theta =\pi +\asy (\nu )\, .
\eee
From this fact and from equations $G_1=0$ and $G_3 =0$ we finally obtain the equations
\be
& G_\pm  + \asy (\nu ) =0   \label {Eq46} \\ 
& E_\pm = \mp 2 p q + \eta (p^{2\sigma +2} +q^{2\sigma +2} ) + \asy (\nu ) \label {Eq47}
\ee
where the asymptotics is uniformly on $D'$, the index $+$ corresponds to the choice $\theta =\asy (\nu )$, the index $-$ corresponds to the choice $\theta =\pi +\asy (\nu )$ and where 
\bee
G_\pm = \pm \left [  (p^2-q^2) \pm \eta pq (p^{2\sigma} - q^{2\sigma }) \right ] \, . 
\eee
The imbalance function $z=p^2-q^2$ is such that 
\bee
p = \sqrt {\frac {1+z}{2}} + \asy (\nu ) \ \ \mbox { and } \ \ q = \sqrt {\frac {1-z}{2}} + \asy (\nu )
\eee
and thus equations (\ref {Eq46}) and (\ref {Eq47}) take the form
\be
& f_\pm (z,\eta ) + \asy (\nu) = 0  \label {Eq48} \\ 
& E_\pm = \mp \sqrt {1-z^2} + \eta \left [ \left ( \frac {1+z}{2} \right )^{\sigma+1} + \left ( \frac {1-z}{2} \right )^{\sigma+1} \right ] + \asy (\nu) \label {Eq49}
\ee
where 
\be
f_\pm (z,\eta ) = z \pm \eta \frac {\sqrt {1-z^2}}{2} \left [ \left ( \frac {1+z}{2} \right )^\sigma - \left ( \frac {1-z}{2} \right )^\sigma \right ] \, . \label {Eq50}
\ee

Since the asymptotic term $\asy (\nu)$ in (\ref {Eq48}), with its derivative with respect to $z$, is uniform with respect to $z \in [-1,+1]$ then it is enough to look for the solutions of equations $f_\pm (z,\eta )=0$.

Of course, equation 
\bee
f_\pm (0,\eta ) = 0
\eee
holds true for any $\eta$; that is the symmetric stationary solution $(z=0 , \theta =0)$ which is positive and the antisymmetric stationary solution $(z=0 , \theta = \pi)$ exist for the nonlinear problem (up to an exponentially small perturbation) as well as for the linear one. 

Since we have assumed, for the sake of definiteness, $\eta >0$; then equation $f_+ (z,\eta )=0$ does not have non zero solutions, indeed the derivative of $f_+$ with respect to $z$ is given by 
\bee
f_+ ' (z,\eta ) = 2 \frac {1+z^2}{[1-z^2]^{3/2}} + \frac 12 \eta \sigma \left [ \left ( \frac {1+z}{2} \right )^{\sigma -1} + \left ( \frac {1-z}{2} \right )^{\sigma -1} \right ] 
\eee
which is always positive for any $z \in [-1 , + 1]$ and for any $\eta >0$.

Thus, we have only to look for the non zero solutions $z$ of equation
\be
f_- (z,\eta )=0 \, . \label {Eq51}
\ee
To this end, we consider the function $\eta (z)$, defined by (\ref {Eq33}), which satisfies the implicit equation 
\bee
f_- \left [ z , \eta (z) \right ] =0\, , \ \forall z \in (0,1)\, .
\eee
Thus, the inverse function $z=z(\eta )$ of $\eta (z)$ gives the solutions of equation (\ref {Eq51}); in order to count the branches of the inverse function $z = z(\eta)$ we compute the first derivative 
\bee
\eta ' (z) = 2^{\sigma +1} \frac {g(z)-g(-z)}{[1-z^2]^{3/2} [(1+z)^\sigma - (1-z)^\sigma ]^2} \, , 
\eee
where
\bee
g(z) = (\sigma z^2 - \sigma z +1) (1+z)^\sigma \, .
\eee
Since 
\bee
\lim_{z\to 0^+} \eta' (z) = 0
\eee
then a bifurcation of the stationary solution occurs at $z=0$ for 
\bee
\eta^\star = \lim_{z\to 0^+} \eta (z) = 2^\sigma /\sigma \, .
\eee
Furthermore, a straightforward calculation gives also that 
\bee
\lim_{z \to 0^+} \eta '' (z) = - \frac {2^\sigma}{3\sigma } (\sigma^2 - 3 \sigma -1)
\eee
and
\be
\lim_{z\to 0} \eta'' (z) 
\left \{
\begin {array}{ll}
> 0 & \ \mbox { if } \ \sigma < \sigma_{threshold} \\
= 0 & \ \mbox { if } \ \sigma = \sigma_{threshold} \\
< 0 & \ \mbox { if } \ \sigma > \sigma_{threshold} \\
\end {array}
\right. \label {Eq52}
\ee
where
\bee
\sigma_{threshold} = \frac {3+\sqrt {13}}{2} \, . 
\eee
Hence, we can conclude that in the case $\sigma \le \sigma_{threshold}$ then we have a supercritical pitchfork bifurcation at $z=0$ (see Fig. \ref {Fig1} - panel (a)), and for $\sigma > \sigma_{threshold}$ then we have a subcritical pitchfork bifurcation at $z=0$ (see Fig. \ref {Fig1} - panel (b)).

Finally, we only have to count the number of branches of the function $z(\eta )$ and thus we look for the number $N$ of the solutions (counting multiplicity) of the equation
\be
h(z)=0 \, , \ h(z)= g(z)-g(-z)\ , \label {Eq53}
\ee
for $z$ in the interval $z \in (-1 , +1)$. 

\begin {lemma} \label {Lemma4}
Let $N$ be the number of solutions $z$ of the equation $h(z)=0$ in the interval $[-1,+1]$, counting multiplicity. \ It follows that $z=0$ is a solution with multiplicity $3$ if $\sigma \not= \sigma_{threshold}$, and with multiplicity $5$ if $\sigma = \sigma_{threshold}$. \ Furthermore, it also follows that 
\be
N = 
\left \{
\begin {array}{ll}
3 & \mbox { if } \sigma < \sigma_{threshold} \\ 
5 & \mbox { if } \sigma \ge \sigma_{threshold}
\end {array}
\right. \, . 
\ee
\end {lemma}

\begin {proof}
We may remark that if $z^\star$ is such that $h(z^\star )=0$ then $h(-z^\star )=0$, too; furthermore $h(\pm 1) = \pm 2^\sigma \not= 0$. \ First of all we see that $z=0$ is a solution of (\ref {Eq53}) with multiplicity $3$ for any $\sigma \not= \sigma_{threshold}$; indeed, a straightforward calculation gives that
\bee
h(0) = h'(0)= h'' (0) =0 \ \ \mbox { and } \ \ h''' (0)= 4 \sigma (-\sigma^2 + 3 \sigma +1) \, . 
\eee
Then $h''' (0) \not= 0 $ if $\sigma \not= \sigma_{threshold}$. \ If $\sigma = \sigma_{threshold}$ then a straightforward calculation gives that $h''' (0) = h^{IV} (0 )=0$ and 
\begin{eqnarray*}
h^V (0) &=& - 86 (\sigma_{threshold}^4 -10 \sigma_{threshold}^3 
+ 20 \sigma_{threshold}^2 - 5 \sigma_{threshold} - 6 ) \\ 
&=& 24 (3+\sqrt {13}) (4+\sqrt {13}) >0.
\end{eqnarray*}

Hence, it follows that
\bee
N \mbox { is } \ \left \{ 
\begin {array}{ll}
\ge 5 & \ \mbox { if } \ \sigma > \sigma_{threshold} \\ 
= 5 \ \mbox { or } \ \ge 9 & \ \mbox { if } \ \sigma = \sigma_{threshold} \\ 
= 3 \ \mbox { or } \ \ge 7 & \ \mbox { if } \ \sigma < \sigma_{threshold}
\end {array}
\right.
\eee
where $N$ is number of solutions, counting multiplicity, of equation $f_- (z,\eta )=0$

Indeed, we see that 
\bee
\lim_{z \to \pm 1} \eta (z) = + \infty \, .
\eee
Then , in the case $\sigma > \sigma_{threshold}$ since $\lim_{z\to 0} \eta'' (z) <0$ then there exists two non-zero solutions of equation (\ref {Eq53}) in the interval $(-1,+1)$ at least; hence, the number $N$ of solutions of equation (\ref {Eq53}), counting multiplicity, is $N \ge 5$. 

In the opposite case $\sigma < \sigma_{threshold}$ it follows $\lim_{z\to 0} \eta '' (z) > 0$, then we have two cases: or equation (\ref {Eq53}) does not have solutions $z\in (-1,+1)$, $z\not= 0$, and in this case $N=3$; or equation (\ref {Eq53}), counting multiplicity, has other solutions $z\in (-1,+1)$, $z\not= 0$, and in this case the number of such a solutions is bigger than $4$, in this case $N \ge 7$. 

Finally, in the case $\sigma = \sigma_{threshold}$ it follows that $\lim_{z\to 0} \eta'' (z)= \lim_{z\to 0} \eta ''' (z) =0$ and 
\bee
\lim_{z\to 0} \eta^{IV} (z) = \frac {6 \cdot 2^{\sigma_{threshold}} \left ( 829 \sqrt {13} + 2989 \right )}{5 \left ( 649 + 180 \sqrt {13} \right ) } >0 \, , 
\eee
hence $N=5$ or $N \ge 9$.

If we can prove that $N \le 5$ then the Theorem is completely proved.

To this end we set
\bee
y = \frac {1-z}{1+z} \, , \ y \in (0, + \infty )\, .  
\eee
Hence, equation $h(z)=0$ in the interval $(-1,1)$ reduces to the equation of the form $p_\sigma (y)=0$ where
\bee
p_\sigma (y) &=& y^\sigma (y^2 + b y +a) - (a y^2 + b y +1)  \\
&=& y^{\sigma +2}  + b y^{\sigma +1}   +a y^\sigma - a y^2 - b y -1
\eee
and where
\bee
a= 1 + 2 \sigma , \ \ b=2 - 2 \sigma \, .
\eee

We remark that if $y^\star >1$ is a root of the polynomial $p_\sigma $ corresponding to a give $z^\star >0$, then $1/y^\star <1$ corresponds to $-z^\star$ and it is a root, too. \ We remark also that: 

\begin {itemize}

\item [-] for any $\sigma \not= \frac {1}{2} \left [ 3 + \sqrt {13} \right ] $ then $p_\sigma$ has solution $y=1$ with multiplicity $3$;

\item [-] for $\sigma = \frac {1}{2} \left [ 3 + \sqrt {13} \right ] $ then $p_\sigma$ has solution $y=1$ with multiplicity $5$.

\end {itemize}

We assume, for a moment, that $\sigma$ is a positive integer and we see that: 

\begin {itemize}

\item [-] $\sigma =1$: in such a case $p_1 = (y-1)^3$ which has only the solution $y=1$ in the interval $(0,+\infty )$ with multiplicity $3$;

\item [-] $\sigma =2$: in such a case $p_2 =(y-1)^3 (y+1)$ which has only the solution $y=1$ in the interval $(0,+\infty )$ with multiplicity $3$;

\item [-] $\sigma = 3$: in such a case $p_3 = (y-1)^3 (y^2-y+1)$ which has only the solution $y=1$ in the interval $(0,+\infty )$ with multiplicity $3$;

\end {itemize}

Now, we are looking for the number $N$ of real solutions, counting multiplicity, of the polynomial $p_\sigma$ in the interval $(0,+\infty )$. \ We already know that for $\sigma =1,2,3$ then $N=3$; we also already know that for $\sigma =4,5,6, \ldots$ then $N\ge 5$ and, in order to get an upper estimate of $N$, we make use of the Budan-Fourier theorem \cite {Prasolov}.

If we denote by $v(y)$ the number of sign changes in the sequence
\bee
\left \{ p_\sigma (y) , p_\sigma' (y), \ldots , p_{\sigma}^{(\sigma )} (y) , p_{\sigma}^{(\sigma +1)} (y) , p_{\sigma}^{(\sigma +2)} (y) \right \}
\eee
then the Budan-Fourier theorem applied to the polynomial $p_\sigma$ with degree $\sigma +2$, where $\sigma =4,5,6, \ldots$, states that
\bee
N \le |v(+\infty ) - v(0)| .
\eee

Since 
\bee
\lim_{y \to + \infty } p_\sigma^{(n)} > 0 
\eee
for any $n=0,1,2, \ldots , \sigma$ then $v (+\infty )=0$. \ On the other hand we observe that
\bee
p_\sigma (0) &=& -1 < 0 \\
p_\sigma ' (0) &=& - b > 0 \\
p_\sigma '' (0) &=& -2  a < 0 \\ 
p_\sigma^{(n)} (0) &=& 0 \ \ \mbox { if } \ 2 < n < \sigma \\
p_\sigma^{(\sigma)} (0) &=& \sigma ! a > 0 \\
p_\sigma^{(\sigma +1)} (0) &=& b (\sigma +1)! < 0 \\
p_\sigma^{(\sigma +2)} (0) &=& a (\sigma +2)! > 0
\eee
since $a=1+2\sigma >0$ and $b= 2-2\sigma <0$ for $\sigma = 4,5,6,\ldots $. \ Then, $v(0)=5$ and so we can conclude that
\bee
N \le | v(+\infty )- v(0)| = 5
\eee
Therefore, the number of solutions $y \in (0,+\infty )$, $y\not= +1$, is exactly equal to $2$.

We prove now that $N \le 5$ even for any positive not integer $\sigma$. \ In order to prove that $N \le 5$ we make use of an extended version of the Budan-Fourier theorem \cite {CLLLR} applied to the polynomial $p_\sigma (y)$ for $y \in (0,+\infty)$. \ If we assume, for a moment, that $\sigma >2$ (the case $1 < \sigma <2$ can be similarly treated) then we set 
\bee
r_0 =0, \ r_1=1, \ r_2=2, \ r_3 = \sigma , \ r_4 = \sigma +1, \ r_5 =\sigma +2 
\eee
where (mimicking Example 2 in \cite {CLLLR})
\bee
a_0 =-1, \ a_1 = (2\sigma -2) , \ a_2 = -(1+2\sigma ),\ a_3 = 1 + 2 \sigma , \ a_4 = -(2\sigma -2) ,\ a_5 =1 . 
\eee 
In this case we have 6 functions $g_j (y)$, $j=0,1,2,3,4,5$ such that
\bee
g_5 (y) = p_\sigma (y)
\eee
and
\bee
\lim_{y\to + \infty} g_j (y) = + \infty 
\eee
and
\bee
g_0 (0) &= & 2 (\sigma +1) (\sigma +2) >0 \\
g_1 (0) &= & -2(\sigma -1)^2 (\sigma +1) \sigma <0 \\
g_2 (y ) &=& (1+2\sigma) \sigma (\sigma -1) (\sigma -1) y^{\sigma -3} + O\left ( y^{\sigma -2} \right ) >0  \ \mbox { as } \ y \to 0^+ \\ 
g_3 (0) &=& -2 (1+2\sigma ) <0 \\ 
g_4 (0) &=& 2( \sigma -1) > 0 \\
g_5 (0) &=& -1 < 0
\eee
Then the sequence $g_j$ has zero sign changes at $+\infty$ and it has $5$ sign changes at $0^+$, i.e.: $v(+\infty )=0$ and $v(0)=5$. \ Therefore,  Theorem 1 \cite {CLLLR} implies that $N \le |v(+\infty )- v(0)| = 5$.

It remains to consider the case $0<\sigma <1$. \ In order to look for the solutions $y>0$ of equation $p_\sigma (y)=0$ we observe that these solutions are such that 
\bee
y^\sigma = \frac {ay^2 + by + 1}{y^2 + b y +a^2}
\eee
where the l.h.s. of this equation is a monotone increasing function, while the r.h.s. is a monotone decreasing function for $0<\sigma <1$. \ Hence, the number of solutions, counting multiplicity, of the equation $p_\sigma (y)=0$ is $N=3$. 
\end {proof}

The proof of the theorem is so completed.

\begin{remark}
From Lemma 4 it turns out that when $\sigma \le \sigma_{threshold}$ then
equation $\eta'(z)=0$ has only solution $z=0$ and therefore, under such
condition on $\sigma$, we only observe a bifurcation of the stationary
solution at $|\eta |= \eta^\star$. On the other side, when $\sigma >  
\sigma_{threshold}$ then the number of solutions (counting multiplicity)
of equation $\eta'(z)=0$ is $5$, since the solution $z=0$ has multiplicity
$3$ then the other $2$ solutions are $\pm z^+$, where $z^+ \in (0,1)$, and
they are associated to saddle points appearing at $|\eta |=\eta^+$, where
$\eta^+ = \eta (z^+)$.
\end{remark}

\begin {remark} We just point out that in the case of $\eta <0$ then we can apply the same arguments; we only have to emphasize that for negative values of $\eta$ then equation $f_- (z,\eta )=0$ does not have non zero solutions and that bifurcations come from equation $f_+ (z,\eta )=0$.
\end {remark}

\begin{figure}
\begin{center}
\includegraphics[height=8cm,width=10cm]{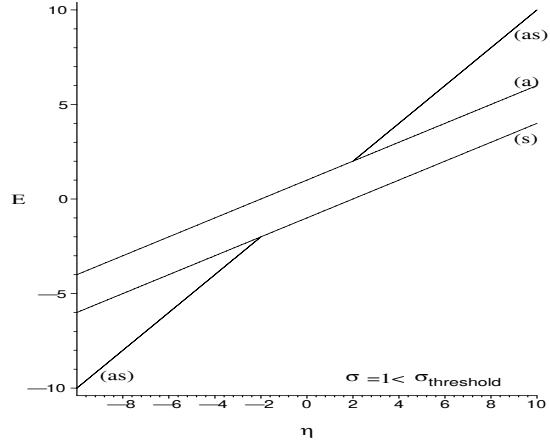}
\end{center}
\caption{\label {Fig2} In this figure we plot the graph of the values of the function $E$ versus the nonlinearity parameter $\eta$ for nonlinearity $\sigma =1 < \sigma_{threshold}$. \ For $\eta = \pm \eta^\star$, $\eta^\star =2$ for $\sigma =1$, a bifurcation occurs and a new branch corresponding to the asymmetrical stationary state appears. \ Line (s) denotes the symmetric stationary solutions, line (a) denotes the antisymmetric stationary solutions, and (as) denote the asymmetrical stationary solutions.}
\end{figure}
\begin{figure}
\begin{center}
\includegraphics[height=8cm,width=10cm]{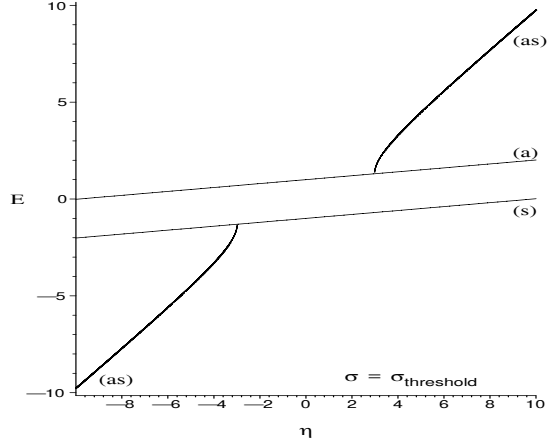}
\end{center}
\caption{\label {Fig3} In this figure we plot the graph of the values of $E$ as function of the nonlinearity parameter $\eta$ for critical  nonlinearity $\sigma =  \sigma_{threshold}$.}
\end{figure}
%\end{center}
%
\begin{figure}
\begin{center}
\includegraphics[height=8cm,width=10cm]{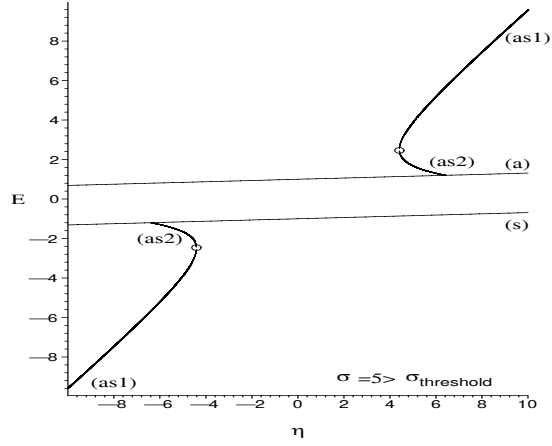}
\end{center}
\caption{\label {Fig4} In this figure we plot the graph of the values of the function $E$ versus the nonlinearity parameter $\eta$ for nonlinearity $\sigma =5 > \sigma_{threshold}$. \ At $|\eta |=\eta^+$, $\eta^+ \approx 4.41$ for $\sigma =5$, a couple of saddle nodes appear, and the corresponding branches, denoted (as1) and (as2), are associated to asymmetrical stationary solutions; asymmetrical solution (as2) then disappears at $|\eta | = \eta^\star$, $\eta^\star = 6.4$ for $\sigma =5$.}
\end{figure}

\begin {remark}
For large $\sigma$ the roots $y < 1$ of the polynomial $p_\sigma (y)$ are asymptotically given by the roots of equation
\bee
(1+2\sigma) y^2 + (2-2\sigma) y +1 =0.
\eee
That is
\bee
y \sim \frac {1}{1+2\sigma} \ \mbox { for } \ \sigma \gg 1 
\eee
Hence, the solution $z^+$ of equation $\eta ' (z)=0$ is asymptotically given by
\bee
z^+ \sim 1 - \frac {1}{\sigma} - \frac {1}{\sigma^2} 
\eee
and we have that
\bee
\eta^+ = \sqrt {2 e \sigma } \left [ 1 + O(\sigma^{-1}) \right ]
\eee
in the limit of large $\sigma$.
\end {remark}

\begin {remark} 
The frequency $\lambda$ of stationary solutions of equation (\ref {Eq25}) are thus given by 
\bee
\lambda = \Omega + \omega E 
\eee 
where $E=E(z)$ is the multivalued function given by (\ref {Eq49}), where $z=z(\eta )$ are the roots of the equation $f_\pm ( z )=0$. \ For the graph of the functions $E(z)$, depending on $\eta$, we refer to the Fig. \ref {Fig2}, Fig. \ref {Fig3} and Fig. \ref {Fig4}. \ We observe the following behaviors (where we assume $\eta <0$ for argument's sake):

\begin {itemize}

\item [-] When $- \eta^\star < \eta < 0$ for $\sigma \le \sigma_{threshold}$, or $- \eta^+ < \eta < 0$ for $\sigma > \sigma_{threshold}$, then we only have the linear stationary states.

\item [-] When $\eta <- \eta^\star$ and $\sigma \le \sigma_{threshold}$, then the symmetric solution bifurcates at $\eta = - \eta^\star$ and then we have 4 stationary solutions: the two linear stationary states and two new asymmetrical stationary states; a similar picture actually occurs also when $\sigma > \sigma_{threshold}$, but in this case the two new asymmetrical stationary solutions don't come by a bifurcation of the symmetric stationary solution, but they come from a branch of saddle points.

\item [-] When $- \eta^\star < \eta  <  - \eta^+$ and $\sigma > \sigma_{threshold}$, then a couple of saddle points occurs and thus we have 4 asymmetrical stationary solutions. \ Two of them, denoted as (as1), are much more localized on a single well than the ones denoted by (as2).

\end {itemize}

\end {remark}

\section {Dynamical stability}\label {Sec4}

The time-dependent equation (\ref {Eq24}), when projected on the one-dimensional spaces spanned by the single-well states $\varphi_R$ and $\varphi_L$, and on the space $\Pi_c L^2 (\R^d )$, takes the form
\be
\left \{
\begin {array}{lcl}
ia_R'  &=& - a_L + r_R  \\ 
i a_L' &=& - a_R + r_L  \\ 
i \psi_c' &=& \frac {1}{\omega} \left [ H_0-\Omega \right ] \psi_c + r_c
\end {array}
\right. \label {Eq55}
\ee
where we have set $\psi \to e^{-i \Omega \tau /\omega} \psi (x,\tau )$. \ We call \emph {two-level approximation} the system of differential equations coming from (\ref {Eq55}) taking $\psi_c =0$ and neglecting the exponential remainder term in $r_{R,L} (a_R, a_L, 0)$ (see Lemma \ref {Lemma2}); in such a case the two-level approximation takes the form 
\be
\left \{
\begin {array}{lcl} 
i a_R' &=& - a_L + \eta  |a_R|^{2\sigma } a_R  \\ 
i a_L' &=& - a_R + \eta  |a_L|^{2\sigma } a_L  
\end {array}
\right. \, ,\ |a_R|^2 + |a_L|^2 =1 
\label {Eq56}
\ee
We may remark that the two-level system (\ref {Eq56}) takes the Hamiltonian form
\bee
i A' = \partial_{\bar A} {\mathcal H} , \ \ A = (a_R, a_L)\, , 
\eee
with Hamiltonian function 
\be
{\mathcal H} = - \left [ \left ( \bar a_R a_L + \bar a_L a_R \right ) - \frac {\eta }{\sigma +1} \left ( |a_R|^{2(\sigma +1)} + |a_L|^{2(\sigma +1)} \right ) \right ] \label {Eq57}
\ee
corresponding to the energy functional restricted to the two-dimensional space spanned by the two single-well states. \ The stationary solutions of the two-level system (\ref {Eq56}) are associated to stationary points of the energy functional ${\mathcal H}$, then we can attribute them some stability/instability properties in the sense of the theory of dynamical system. \ In particular, let 
$\theta = \mbox {arg} (a_R) - \mbox {arg} (a_L)$ be the difference between the phases of $a_R$ and $a_L$, and let $z=|a_R|^2-|a_L|^2$ be the imbalance function, then system (\ref {Eq56}) takes the Hamiltonian form 
\be
\left \{
\begin {array}{lcl}
\dot \theta &=& \partial_z {\mathcal H} \\ 
\dot z &=& - \partial_\theta {\mathcal H} 
\end {array}
\right. \label {Eq58}
\ee
where the Hamiltonian (\ref {Eq57}) takes now the form
\bee
{\mathcal H} = - \sqrt {1-z^2} \cos \theta + \frac {\eta}{\sigma +1} \left [ \left ( \frac {1+z}{2} \right )^{\sigma +1} + \left ( \frac {1-z}{2} \right )^{\sigma +1} \right ] \, . 
\eee

In order to study the stability properties of the stationary solutions of equation (\ref {Eq58}) we have to consider the matrix
\bee
Hess = \left ( 
\begin {array}{cc}
\frac {\partial^2 {\mathcal H}}{\partial z \partial \theta} & \frac {\partial^2 {\mathcal H}}{\partial z^2} \\ 
- \frac {\partial^2 {\mathcal H}}{\partial \theta^2} & -\frac {\partial^2 {\mathcal H}}{\partial \theta \partial z} 
\end {array}
\right )
\eee
at the stationary points. \ Since the trace of $Hess$ is zero then we have that the stationary point is a circle if det $Hess >0$, and it is a saddle point if det $Hess <0$. 

\subsection {Dynamical stability of the symmetric and antisymmetric stationary states} We consider, at first, the symmetric and antisymmetric stationary states corresponding to $\theta =0$ and $z=0$ (symmetric), and $\theta =\pi $ and $z=0$ (antisymmetric). \ A straightforward calculation gives that 
\bee
\left. \mbox {det } Hess \right |_{\theta=0, \ z=0} = 1 + \eta \frac {\sigma}{2^\sigma} \ \mbox { and } \ \left. \mbox {det } Hess \right |_{\theta=\pi , \ z=0} = 1 - \eta \frac {\sigma}{2^\sigma}\, .
\eee
Then, it follows that the symmetric stationary solution is dynamically stable for any $\eta > - \eta^\star$, and the antisymmetric stationary solution is dynamically stable for any $\eta < \eta^\star$, where $\eta^\star = 2^\sigma/\sigma$.

\subsection {Dynamical stability of the asymmetrical stationary solutions.} \ For argument's sake let us assume  $\eta <0$. \ Then the symmetric stationary solution bifurcates and new asymmetrical solutions appear, they correspond to $\theta =0$ and the values of $z$ are the non zero solutions of the equation $f_+ (z,\eta )=0$ (in fact, we have assumed $\eta <0$; in the case of $\eta >0$, as considered in \S \ref {Sec3} for the sake of definiteness, then the stationary solutions corresponds to the roots $z$ of equation $f_- (z,\eta )=0$). \ A straightforward calculation gives that 
\bee
\left. 
\mbox {det } Hess 
\right |_{\theta =0} 
= \sqrt {1-z^2} \left [ (1-z^2)^{-3/2} + \frac {\eta \sigma}{4} \left (  
\left ( \frac {1+z}{2} \right )^{\sigma -1} + \left ( \frac {1-z}{2} \right )^{\sigma -1}  \right ) \right ]
\eee
By the relation $\eta =\eta (z)$ implicitly defined by the equation $f_+ (z,\eta )=0$ it follows that
\bee
\left. \mbox {det } Hess \right |_{\theta =0, \ \eta = \eta (z) } =  \frac {g(z)-g(-z)}{(1-z^2) \left [ (1+z)^\sigma - (1-z)^\sigma \right ] }
\eee
where it has been already proved that the equation $g(z)-g(-z)=0$ has a solution at $z=0$ with multiplicity $3$ (multiplicity $5$ if $\sigma=\sigma_{threshold}$). \ Since this equation has no other solution for $\sigma \le \sigma_{threshold}$, since $q(z)=q(-z)$ and since 
\bee
\lim_{z \to 1^-} \mbox {det} \left. Hess \right |_{\theta =0, \ \eta = \eta (z) } = + \infty 
\eee
then 
\bee
\mbox {det} \left. Hess \right |_{\theta =0, \ \eta = \eta (z) } > 0 \, , \ \forall z \not= 0 .
\eee
Then, the asymmetrical solutions, if there, are stable. \ On the other side, for $\sigma > \sigma_{threshold}$ then the equation $g(z)-g(-z)=0$ has three distinct solutions; hence, by means of the same arguments as before, it follows that the branch $(as2)$ is dynamically unstable and the branch $(as1)$ is dynamically stable. 

We can collect all these results as follows (see also Fig. \ref {Fig1}).

\begin {theorem} \label {Theorem2} 
Let us consider the stationary solutions of the {\bf  two level approximation} (\ref {Eq56}) that coincide, up to an exponentially small term, with the solutions given in Theorem \ref {Theorem1}. \ The symmetric and antisymmetric solutions of the two-level approximation are such that:

\begin {itemize}

\item [-] for any $\sigma >0$, the symmetric stationary solution $(s)$ is stable for any $\eta \ge - \eta^\star$, and it is unstable for any $\eta < - \eta^\star$;

\item [-] for any $\sigma >0$, the antisymmetric stationary solution $(a)$ is stable for any $\eta \le  \eta^\star$, and it is unstable for any $\eta > \eta^\star$.

\end {itemize}

The asymmetrical solutions of the two-level approximation are such that:

\begin {itemize}

\item [-] for any $\sigma \le \sigma_{threshold}$ the asymmetrical stationary solution $(as)$ is stable;

\item [-] for any $\sigma > \sigma_{threshold}$ the branch $(as2)$ of the asymmetrical stationary solution there exists for any $\eta^+ < |\eta |< \eta^\star$ and it is unstable, the other branch $(as1)$ of the asymmetrical stationary solution there exists for any $\eta^+ < |\eta |$ and it is stable.

\end {itemize}

\end {theorem}

\section {Orbital stability} \label {Sec5}
In this section our aim is to study the orbital stability of the stationary solutions of the NLS (\ref {Eq1}). \ So far we have considered both cases of attractive and repulsive nonlinearity for any couple of eigenvalues $\lambda_\pm$. \ Hereafter we consider only the first two eigenvalues and we assume to be in the attractive nonlinearity, that is:

\begin {hypothesis} \label {Hyp4} Let $\lambda_\pm$ be the {\bf first two eigenvalues} of $H_0$. \ Let $\eta = \frac {\epsilon}{\omega} \langle \varphi_R^{\sigma +1}, g \varphi_R^{\sigma +1} \rangle$ be the effective nonlinearity parameter in Eq.(\ref {Eq14}) where $\langle \varphi_R^{\sigma +1}, g \varphi_R^{\sigma +1} \rangle >0$; we assume that
\bee
\epsilon < 0 \ \ \mbox { that is } \ \ \eta < 0 \, .
\eee
\end {hypothesis}

If we rescale the solution $\psi$ as $\phi = |\epsilon |^{1/2\sigma} \psi$, then equation (\ref {Eq1}) is equivalent to the equation
\be
i \hbar \frac {\partial \phi }{\partial t} 
= H_0 \phi - g |\phi |^{2\sigma} \phi, 
\quad  \| \phi \| =|\epsilon |^{1/2\sigma} . \label {Eq59}
\ee
The stationary solutions of the equation
\be
H_0 \phi_{\lambda ,\epsilon } - g |\phi_{\lambda ,\epsilon } |^{2\sigma } \phi_{\lambda ,\epsilon } - \lambda \phi_{\lambda,\epsilon } =0, \quad  \lambda = \Omega + \omega E \, , \label {Eq60}
\ee
are associated, by means of the scaling, to the stationary solutions $\psi_E^s$, $\psi_E^a$ and $\psi_E^{as}$ given in Theorem \ref {Theorem1} where $E=E(\epsilon )$ is a multivalued function and where the stationary solutions are now denoted by
\bee
\phi_{\lambda ,\epsilon }^s:& \mbox { symmetric stationary solution} \\ 
\phi_{\lambda ,\epsilon }^a:& \mbox { antisymmetric stationary solution} \\ 
\phi_{\lambda ,\epsilon }^{as}:& \mbox { asymmetrical stationary solution for } \sigma \le \sigma_{threshold} \\ 
\phi_{\lambda ,\epsilon }^{as1} \mbox { and } \ \phi_{\lambda ,\epsilon }^{as2} :& \mbox { asymmetrical stationary solutions for } \sigma > \sigma_{threshold}
\eee
If we consider a general stationary state, we denote the solution by $\phi_{\lambda,\epsilon}$ and $\psi_E$, but if we want to distinguish the branches, we insist, in such above way, by denoting $s$, $a$, $as$, $as1$ and $as2$, on each shoulder of solutions. 

Here, we consider the orbital stability for the symmetric stationary solution $\phi^s_{\lambda ,\epsilon }$ and for the asymmetrical stationary solutions $\phi^{as}_{\lambda ,\epsilon }$ that bifurcate from the symmetric one.

\begin {definition}
The family of nonlinear bound states $\{e^{i\alpha }\phi_{\lambda, \epsilon}, \alpha \in \R\}$ is said to be orbitally stable in $H^1(\R^d)$ if for any $\kappa>0$ there exists a $\delta>0$ such that if $\phi_0$ satisfies 
\be
\inf_{\alpha\in \R} \|\phi_0-e^{i\alpha}\phi_{\lambda, \epsilon}\|_{H^1} <\delta, \label {Eq61}
\ee 
then for all $t\ge 0$, the solution $\phi(t)$ of (\ref{Eq59}) with $\phi(0)=\phi_0$ exists and satisfies 
\bee 
\inf_{\alpha \in \R} \|\phi(\cdot,t) -e^{i\alpha}\phi_{\lambda, \epsilon}\|_{H^1}  <\kappa.
\eee
Otherwise, it is said to be unstable in $H^1(\R^d)$. 
\end {definition}

The main result of this section is the following:

\begin {theorem} \label{Theorem3} 
Fix any $\hbar>0$ be sufficiently small such that $\hbar \in (0,\hbar_3)$ for some $\hbar_3 >0$ small enough. \ Then, the following statements hold.

\begin{itemize}

\item Let $\sigma \le \sigma_{threshold}$. \ The symmetric solution corresponding to $z^{s}=\tilde{O}(e^{-\rho/\hbar})$ is orbitally stable in $H^1$ for $|\eta| < \eta^\star$. \ At the bifurcation point $\eta=\eta^\star$, there is an exchange of stability, that is, for $|\eta| > \eta^\star$, the asymmetric solution is stable in $H^1$ and the symmetric solution is unstable.

\item Let $\sigma > \sigma_{threshold}$. By Theorem \ref {Theorem1}, two couples of new asymmetric stationary states, denoted by $\psi^{as1}$ and $\psi^{as2}$ appears at $|\eta|=\eta^+$. \ For $|\eta| > \eta^+$, $\psi^{as1}$ is orbitally stable in $H^1$, $\psi^{as2}$ is unstable. \ 
On the other hand, 
the symmetric state is orbitally stable in $H^1$ for $|\eta|<\eta^\star,$ 
and unstable for $|\eta|>\eta^\star$.  

\end{itemize}

\end {theorem}

As a standard method to prove the orbital stability of a stationary solution $\phi_{\lambda , \epsilon}$, the following proposition is well known. \ We first define $L_+^{\lambda , \epsilon }$ and $L_-^{\lambda,\epsilon}$, which are respectively the real and the imaginary part of the linearized operators around a real valued stationary solution $\phi_{\lambda , \epsilon}$ :  
\bee
L_+^{\lambda , \epsilon } \equiv L_+ [\phi_{\lambda, \epsilon}] 
= H_0 - \lambda - (2\sigma +1) 
g |\phi_{\lambda ,\epsilon } |^{2\sigma },
\eee
\bee
L_-^{\lambda,\epsilon}\equiv L_- [\phi_{\lambda, \epsilon}] 
= H_0 - \lambda - g |\phi_{\lambda ,\epsilon } |^{2\sigma }.  
\eee
It is clear that $L_-^{\lambda,\epsilon} \phi_{\lambda, \epsilon} =0$ since $\phi_{\lambda, \epsilon}$ is a solution of (\ref{Eq60}). \ 
Moreover, $L_+^{\lambda,\epsilon}$ and $L_-^{\lambda,\epsilon}$ are self-adjoint operators on $L^2(\R^d)$ with domain $H^2(\R^d)$. \ 
The essential spectrum of these two operators coincides with the interval $[V_{\infty}^- -\lambda, \infty)$ with $V_{\infty}^- -\lambda >0$, since $\phi_{\lambda, \epsilon}$ vanishes at infinity; indeed, $V$ is bounded, and we can apply the proof of Theorem 1 in \cite{FuOz}, regarding the term 
$V \phi_{\lambda,\epsilon}$ of (\ref{Eq60}) as one of nonlinear parts. \ There are also finitely many of discrete spectrum and  $\sigma_d(L_{\pm}^{\lambda,\epsilon}) \subset (-\infty, V_{\infty}-\lambda)$ (see \cite{BeSh}).  

In order to prove the orbital stability we make use of the following criteria (see, e.g., \cite{G} or Part I of \cite{GSS}).

\begin{proposition} \label{Proposition1}  
%Assume that $\sigma$ satisfies (\ref{Eq21}). 
\ Suppose that $L_-^{\lambda,\epsilon}$ is nonnegative. 
\ Let $F(\lambda ) = \| \phi_{\lambda ,\epsilon } \|^2$. 

\begin{itemize}

\item[(1)] If $L_+^{\lambda,\epsilon}$ has only one negative eigenvalue, and $dF/d\lambda <0,$ then, $\phi_{\lambda ,\epsilon}$ is stable in $H^1(\R^d)$. 

\item[(2)] If $L_+^{\lambda,\epsilon}$ has only one negative eigenvalue, and $dF/d\lambda >0,$  then, $\phi_{\lambda ,\epsilon}$ is unstable in $H^1(\R^d)$. 

\item[(3)] If $L_+^{\lambda,\epsilon}$ has at least two negative eigenvalues, then, $\phi_{\lambda ,\epsilon}$ is unstable in $H^1(\R^d)$. 
 
\end{itemize}
\end{proposition}

\begin{remark}
For the instability (3), it is enough to find a vector $p\in H^1$ such that 
\begin{equation} \label{Eq62}
\langle L_+^{\lambda,\epsilon} p,p \rangle <0, \quad p \perp \phi_{\lambda,\epsilon} \mbox{~in~} L^2. 
\end{equation}
(see for e.g., \cite{CoCoOhIHP,G}). \ As we will see below, ``$L_-^{\lambda,\epsilon}$ is nonnegative and $L_+^{\lambda,\epsilon}$ has two negative eigwnvalues'' occurs only for the symmetric stationary solution $\phi_{\lambda,\epsilon}^s$. \ In this case, we can find the normalized antisymmetric solution $\frac{\phi_{\lambda,\epsilon}^a}{\|\phi_{\lambda,\epsilon}^a\|}$ as the vector $p$ satisfying the property (\ref{Eq62}) for $\hbar$ small.   
\end{remark}

We shall therefore check the following properties:

\begin{itemize}

\item the number of negative eigenvalues of $L_+^{\lambda,\epsilon}$;

\item $L_-^{\lambda,\epsilon}$ is a nonnegative operator;

\item (Slope condition) the sign of the function $d F (\lambda )/d\lambda $.  

\end{itemize}

\subsection {Number of negative eigenvalues of $L_+^{\lambda,\epsilon}$}

First we consider the number of negative eigenvalues of $L_+^{\lambda,\epsilon}$. \ We will prove that:

\begin{lemma} \label{Lemma5} 
%Assume that $\sigma$ satisfies (\ref{Eq21}). 
\ Let $\hbar^\star >0$ small enough as in Theorem \ref {Theorem1}; there exists $\hbar_1 \in (0, \hbar^\star )$ such that for any $\hbar \in (0, \hbar_1)$ the following statements are satisfied.

\begin {itemize} 

\item [(i)] Let $\lambda$ be the energy level associated to the symmetric stationary state $\phi_{\lambda , \epsilon } 
= \phi^s_{\lambda , \epsilon }$. \ Then, $L_+^{\lambda,\epsilon}$ admits only one negative eigenvalue provided that $|\eta | < \eta^\star$. \ On the other hand, $L_+^{\lambda,\epsilon}$ admits two negative eigenvalues provided that $|\eta | > \eta^\star$.  

\item [(ii)] Let $\lambda$ be the energy level associated to the asymmetrical stationary state $\phi_{\lambda , \epsilon }= \phi^{as}_{\lambda ,\epsilon }$ if $\sigma \le \sigma_{threshold}$, and $\phi_{\lambda , \epsilon } = \phi^{as1}_{\lambda , \epsilon }$ and $\phi_{\lambda , \epsilon } = \phi^{as2}_{\lambda , \epsilon }$, if $\sigma > \sigma_{threshold}$. \ Then, $L_+^{\lambda,\epsilon}$ admits only one negative eigenvalue. 

\end{itemize} 

\end{lemma}

\begin {proof} 
We set 
\bee
&&
\phi_{\lambda ,\epsilon } = 
a_R^{\lambda , \epsilon} \varphi_R + a_L^{\lambda , \epsilon} \varphi_L 
+ \phi_c^{\lambda , \epsilon}, \quad 
| a_R^{\lambda , \epsilon} |^2 + |a_L^{\lambda , \epsilon}|^2 
+ \| \phi_c^{\lambda , \epsilon} \|^2 = |\epsilon |^{1/\sigma }, \\
&& 
\| \phi_c^{\lambda , \epsilon} \| 
= |\epsilon |^{1/2\sigma } \| \psi_c \|, 
\quad 
\psi_c=\psi_E-(a_R^\lambda \varphi_R +a_L^\lambda \varphi_L),
\eee
where $\psi_E$ is a stationary solution obtained in Theorem \ref{Theorem1}. 

We consider the eigenvalue problem $L_+^{\lambda,\epsilon} u = (\omega \mu ) u$ with $u \in H^2(\R^d)$ and where
\be
|\mu \omega |\le C \hbar^2 \, . \label {Eq63}
\ee
By setting $u=a_R \varphi_R +a_L \varphi_L + u_c$ with $u_c \in \Pi_c L^2$, then the eigenvalue problem takes the following form
\be \label{Eq64}
\left \{
\begin {array}{lcl}  
\omega \mu a_R &=& a_R \Omega -a_L \omega - \lambda a_R - (2\sigma +1)  
\langle \varphi_R , g |\phi_{\lambda , \epsilon }|^{2\sigma} u \rangle \\ 
\omega \mu a_L &=& a_L \Omega - a_R \omega - \lambda a_L - (2\sigma +1) 
\langle \varphi_L , g |\phi_{\lambda , \epsilon }|^{2\sigma} u \rangle \\ 
\omega \mu u_c &=& (H_0 - \lambda ) u_c - \Pi_c (2\sigma +1) 
g |\phi_{\lambda , \epsilon }|^{2\sigma} u 
\end {array}
\right. \, . 
\ee
The last equation reads as 
\bee 
&&\left [I - \left [ H_0 - \lambda - \omega \mu \right ]^{-1} 
\Pi_c (2\sigma +1) g |\phi_{\lambda , \epsilon }|^{2\sigma} \right ] u_c \\
&& = \left(H_0 - \lambda - \omega \mu \right )^{-1}
\Pi_c (2\sigma +1) g |\phi_{\lambda , \epsilon }|^{2\sigma}(a_R \varphi_R +a_L \varphi_L) 
\eee
Since $H_0 - \lambda \ge C \hbar$, when restricted to $\Pi_c L^2$, and since (\ref {Eq63}), then we have  
\bee
\|\left( H_0 - \lambda - \omega \mu \right )^{-1}\Pi_c \|_{\mathcal{L}(L^2\to H^2)} \le C_1 \hbar^{-1} \, . 
\eee
Here, we recall that, from (\ref{Eq20}) and (\ref{Eq36}), 
\bee
\| g |\phi_{\lambda , \epsilon }|^{2\sigma} \| 
= |\epsilon | \| g |\psi_{E(\epsilon)}|^{2\sigma} \| 
\le C |\epsilon | \hbar^{1-\alpha_0} 
\eee
with $\alpha_0=1+\frac{d(2\sigma-1)}{4}$. \ Thus, if $\omega \mu =O(\hbar^2),$ we get from (\ref{Eq17}), for sufficiently small $\hbar,$ 
\bee
\| \left(H_0 - \lambda - \omega \mu \right)^{-1} 
\Pi_c (2\sigma +1) g |\phi_{\lambda , \epsilon }|^{2\sigma}\|_{\mathcal{L}(L^2\to H^2)} 
\le C_2 (2\sigma+1) |\epsilon| \hbar^{-\alpha_0} \le \frac{1}{2}. 
\eee
Namely, if $\mu $ satisfies the condition (\ref {Eq63}) then the inverse of the operator 
\bee
I - \left [ H_0 - \lambda - \omega \mu \right ]^{-1} \Pi_c (2\sigma +1) g |\phi_{\lambda , \epsilon }|^{2\sigma}
\eee
exists. \ Accordingly, the third equation in (\ref {Eq64}) has a solution  
\bee
u_c &:= & u_c ( \mu, \lambda ) \\ 
&=& Q[\mu, \phi_{\lambda,\epsilon}](a_R \varphi_R +a_L \varphi_L),
\eee
where 
\bee
Q[\mu, \phi_{\lambda,\epsilon}]
&=&
\left [I - \left( H_0 - \lambda - \omega \mu \right)^{-1} 
\Pi_c (2\sigma +1) g |\phi_{\lambda , \epsilon }|^{2\sigma} \right ]^{-1} \\
&& \times \left( H_0 - \lambda - \omega \mu \right)^{-1} 
\Pi_c (2\sigma +1) g |\phi_{\lambda , \epsilon }|^{2\sigma} : L^2(\R^d) \to H^2(\R^d), 
\eee
and 
\begin{equation*}
\|Q[\mu, \phi_{\lambda,\epsilon}]\|_{\mathcal{L}(L^2\to H^2)} \le C_{\sigma}|\epsilon| \hbar^{-\alpha_0}. 
\end{equation*}
The bound $C_{\sigma}$ is uniform in $\hbar$ on $D$, where $D$ is defined in Lemma \ref{Lemma3}, and for any $\mu$ such that $|\mu \omega | \le C \hbar^2$. \ In fact, by the same arguments the same estimate holds true also for the derivative of $Q$ with respect to $\mu$:
\begin{equation} \label{Eq65} 
\left \| \frac {\partial Q}{\partial \mu} \right \|_{\mathcal{L}(L^2\to H^2)} \le C |\epsilon | \hbar^{-\alpha_0'}
\end{equation}
for some $\alpha_0' >0$. \ We insert this expression of $u_c$ into the system (\ref{Eq64}), and we have Lyapunov-Schmidt reduction of (\ref{Eq64}) as follows. 
\bee
\left \{
\begin {array}{lcl}
\omega \mu a_R &=& a_R \Omega - a_L \omega - \lambda a_R \\
&& - (2\sigma +1)  
\langle \varphi_R , g |\phi_{\lambda , \epsilon }|^{2\sigma} 
(I+Q(\mu, \phi_{\lambda, \epsilon}))(a_R \varphi_R +a_L \varphi_L) \rangle, \\ 
\omega \mu a_L &=& a_L \Omega - a_R \omega - \lambda a_L \\
&& - (2\sigma +1) 
\langle \varphi_L , g |\phi_{\lambda , \epsilon }|^{2\sigma} 
(I+Q(\mu, \phi_{\lambda, \epsilon}))(a_R \varphi_R +a_L \varphi_L) \rangle.
\end {array}
\right. 
\eee
This system can be rewritten under the following form. 
\be \label{Eq66}  
(N+\mu I - \nu C) 
\begin{pmatrix}
a_R \\
a_L
\end{pmatrix}
=
\begin{pmatrix}
0\\
0 
\end{pmatrix}
,
\ee
where we recall that $\lambda = \Omega + \omega E$ and where 
\bee
&& N=
\begin{pmatrix}
\alpha, & 1 \\
1, & \beta 
\end{pmatrix}
, \quad I=
\begin{pmatrix}
1 & 0 \\
0 & 1 
\end{pmatrix}
, \quad C=
\begin{pmatrix}
C_1, & C_2 \\
C_3, & C_4 
\end{pmatrix}
,\\
&& \alpha = E + (2\sigma +1) |\eta | |a_R^\lambda |^{2\sigma}, \quad 
\beta = E + (2\sigma +1) |\eta | |a_L^\lambda |^{2\sigma},
\eee
and $\nu=e^{-\gamma \rho' /\hbar}$ for any $\rho' \in (0,\rho)$ as in (\ref{Eq43}). \ For $\nu \ne 0$, we have put
\bee
C_1 &=&C_1(a_R^{\lambda}, a_L^{\lambda}, E, \mu;\hbar) =-\frac{(2\sigma+1)|\epsilon|}{\nu \omega}
\Big\{\langle \varphi_R, g(|\psi_E|^{2\sigma}-|a_R^{\lambda} \varphi_R|^{2\sigma})
\varphi_R \rangle \\
&& \hspace{6cm}+ \langle \varphi_R, g|\psi_E|^{2\sigma}Q(\mu,\phi_{\lambda, \epsilon})\varphi_R \rangle\Big\},\\
C_2 &=&C_2(a_R^{\lambda}, a_L^{\lambda}, E, \mu;\hbar) = -\frac{(2\sigma+1)|\epsilon|}{\nu \omega}
\Big\{\langle \varphi_R, g|\psi_E|^{2\sigma} (I+Q(\mu,\phi_{\lambda, \epsilon}))\varphi_L \rangle \Big\}, \\
C_3 &=& \bar{C_2},\\ 
C_4 &=& C_4(a_R^{\lambda}, a_L^{\lambda}, E, \mu;\hbar) = -\frac{(2\sigma+1)|\epsilon|}{\nu \omega}
\Big\{\langle \varphi_L, g(|\psi_E|^{2\sigma}-|a_L^{\lambda} \varphi_L|^{2\sigma}) \varphi_L \rangle \\
&& \hspace{6cm}+ \langle \varphi_L, g|\psi_E|^{2\sigma}Q(\mu,\phi_{\lambda, \epsilon})\varphi_L \rangle\Big\}.
\eee
%
%We may observe that the entries $C_j$ of the matrix $C$ are $C^1$ in $\mu$, since %$(H_0-\lambda-\omega\mu)^{-1}$ is analytic in $\mu$.

If $\nu=0,$ then $\mu$ are the eigenvalues of $N$ and they are the solutions of the equation 
\bee
P(\mu)=\mu^2 + (\alpha +\beta ) \mu + \alpha \beta -1 = 0,
\eee
which always has only two real different solutions $\mu_1$, $\mu_2$ since $(\alpha +\beta )^2 - 4 \alpha \beta + 4 = (\alpha -\beta )^2 + 4 >0$. \ In particular, these two real eigenvalues are both negative or both positive if $\alpha \beta >1$, or only one is negative in $\alpha \beta < 1$. 

To investigate the sign of $\alpha \beta -1$, we consider, at first, the case of the symmetric stationary solution corresponding to 
$z^\lambda =z^s = 0 $ (see Theorem \ref {Theorem1} and Remark \ref {Remark7}). \ Then (hereafter, for the sake of simplicity, 
we denote by $\sim$ that we have an exponentially small term)
\bee
a_R^\lambda = a_L^\lambda = \frac {1}{\sqrt {2}}\, , \ \ E \sim -1-|\eta| \frac {1}{2^\sigma}
\eee
and
\bee
\alpha = \beta = E + |\eta | \frac {2\sigma +1}{2^\sigma} \sim -1 + |\eta | \frac {2\sigma }{2^\sigma}
\eee
Hence, condition $\alpha \beta >1$ is equivalent to the condition $|\eta | > \eta^\star = \frac {2^\sigma }{\sigma}$ (and in such a case both solutions are negative), and condition $\alpha \beta <1$ is equivalent to the condition $|\eta | < \eta^\star = \frac {2^\sigma }{\sigma}$; provided $\hbar$ is small enough.

We consider next the case of the asymmetrical stationary solution corresponding to $z^\lambda  \not= 0$. \ In such a case we set $a= \frac 12 |\eta |(p^\lambda )^{2\sigma} $ and $b= \frac 12 |\eta | (q^\lambda)^{2\sigma}$, then 
\bee
\alpha \sim E + 2 (2\sigma +1 ) a \, , \ \ \beta \sim E + 2 (2\sigma +1 ) b 
\eee
and
\bee  
E \sim - \sqrt {1-(z^\lambda)^2} - 2 [a(p^\lambda)^2 + b (q^\lambda)^2 ]
\eee
Hence, condition $\alpha \beta < 1$ is equivalent to the condition
\bee
\ell (z^\lambda ,\sigma) -1 < 0
\eee
where
\bee
\ell (z,\sigma ) &:=& \left [ \sqrt {1-z^2} + 2 \left [ (p^2 - 2 \sigma -1 ) a + q^2 b \right ] \right ] \times \\
&& \hspace{30mm}
\left [ \sqrt {1-z^2} + 2 \left [ ap^2  + b (q^2 - 2 \sigma -1 ) \right ] \right ] \\ 
&& \hspace{-2cm}\sim \frac {\left [ (-1+z+4z\sigma ) (1+z)^\sigma + (1-z)^{\sigma +1} \right ] }
{(1-z^2) \left [ (1-z)^\sigma - (1+z)^\sigma \right ]^2} 
\times \left [ (1+z+4z\sigma ) (1-z)^\sigma - (1+z)^{\sigma +1} \right ]
\eee
since
\bee
p\sim \sqrt {\frac {1+z}{2}} \ \mbox { and } \ q\sim \sqrt {\frac {1-z}{2}} \, . 
\eee
Then, a straightforward calculation gives that 
\bee
\ell (z,\sigma ) -1 &\sim & \frac{4 z \sigma (1+z)^{2\sigma }}{(1-z^2) \left [ (1+z)^\sigma - (1-z)^\sigma \right ]^2 } \\
&& \hspace{20mm}\times \Big[ -z-1 - (z-1)\frac {(1-z)^{2\sigma}}{(1+z)^{2\sigma}} + 2z(1+2\sigma )
\frac{(1-z)^{\sigma}}{(1+z)^{\sigma}} \Big]  \\ 
&=& \frac {4 z \sigma (1+z)^{2\sigma }}{(1-z^2) \left[(1+z)^\sigma - (1-z)^\sigma \right]^2}\\ 
&& \hspace{20mm}\times\frac{2}{1+y} \left[ -1-y^{2\sigma+1} + (1+2\sigma ) y^\sigma - (1+2\sigma ) y^{\sigma +1} \right ] 
\eee
where we have set $\displaystyle{y=\frac {1-z}{1+z} \in [0,1]}$. \ We then consider the sign of the following polynomial in the right hand side above. 
\bee
q(y):=-1 + y^{2\sigma+1} + (1+2\sigma ) y^\sigma - (1+2\sigma ) y^{\sigma +1}.  
\eee
It is in fact easy to conclude that $q(y) \le 0$ for any $y \in [0,1]$. \ Indeed,  
\bee
q(y) \le -1 + (1+2\sigma ) y^\sigma - (1+2\sigma ) y^{\sigma +1} \le 0.
\eee

Now, we wish to investigate the sign of eigenvalues for the case $\nu \ne 0.$  \ Recall that the effective nonlinearity parameter $\eta$ 
satisfies $|\eta| \le C$ for some constant $C>0$. \ Also there exist $\hbar_0 \in (0,\hbar^\star)$, and a compact interval $K_{\hbar_0}$ such that the two eigenvalues of the matrix $N$,  
\bee 
\mu_1=\frac{1}{2}\{-(\alpha+\beta) -\sqrt{(\alpha-\beta)^2 +4} \}, \quad 
\mu_2=\frac{1}{2}\{-(\alpha+\beta) +\sqrt{(\alpha-\beta)^2 +4} \}
\eee
belong to $K_{\hbar_0}$ for any $\hbar \in (0, \hbar_0)$. \ Then we see that $C_j=C_j(a_R^{\lambda}, a_L^{\lambda}, E, \mu, \hbar)$ 
are bounded, together with their first derivatives, on $D \times K_{\hbar_0}$ uniformly for any $\hbar \in (0,\hbar_0)$: 
indeed, there exists a constant $C>0$ such that  
\begin{eqnarray*}
&& \nu^{-1}|\langle \varphi_L, g|\psi_E|^{2\sigma}
(I+ Q(\mu, \phi_{\lambda,\epsilon}))\varphi_R \rangle |\\ 
&&\le \nu^{-1} \Big[\|g\|_{L^{\infty}}
 \|\varphi_R \varphi_L\|_{L^{\infty}} \|\psi_E\|_{L^{2\sigma}}^{2\sigma}
+\|g\|_{L^{\infty}}\|\phi_L\| \|\phi_R\|_{L^4}^2 \|\psi_E\|_{L^{8\sigma}}^{4\sigma}\Big], 
\end{eqnarray*}
and this right hand side is bounded because of (\ref{Eq14}), (\ref{Eq20}) and (\ref{Eq36}). \ 
It also follows that if $1 \le 2\sigma$, 
\bee
\nu^{-1}|\langle \varphi_R, g(|\psi_E|^{2\sigma}-|a_R^{\lambda} \varphi_R |^{2\sigma}) \varphi_R \rangle | 
\le \nu^{-1} C(\|\varphi_R \varphi_L\| + \| \psi_c\|) (1+\|\psi_E\|_{L^{2(2\sigma-1)}}^{2\sigma-1}), 
\eee
whose right hand side is bounded, noting (\ref{Eq20}), (\ref{Eq21}), (\ref{Eq36}) and (\ref{Eq34}). \ If $0<2\sigma <1$, 
\bee
&& |\langle \varphi_R, g(|\psi_E|^{2\sigma}-|a_R^{\lambda} \varphi_R|^{2\sigma}) \varphi_R \rangle | 
\le   C \int \varphi_R^2 |g| \, | a_L^\lambda \varphi_L + \psi_{c} |^{2\sigma} dx \\ 
&& \ \ \ \ \le  C \| g \|_{L^\infty} \int \varphi_R^2  \, | \psi_{c} |^{2\sigma} dx +  \| g \|_{L^\infty} | a_L^\lambda |^{2\sigma} \int \varphi_R^2   \varphi_L^{2\sigma} dx.
\eee 
The first integral is estimated as follows
\be
\int \varphi_R^2  \, | \psi_{c} |^{2\sigma} dx \le \| \varphi_R^2 \|_{L^p} \cdot \| | \psi_{c} |^{2\sigma} \|_{L^q} 
= \| \psi_{c}  \|^{2\sigma} \| \varphi_R \|_{L^{2/(1-\sigma)}}^{2} \label {Eq67}
\ee
by means of the H\"older inequality, where $q= \frac {1}{\sigma} > 2$ and $p= \frac {1}{1-\sigma}$. Inequalities (\ref{Eq34}) and (\ref{Eq14}) yield that this right hand side is exponentially small. \ Similarly, the estimate of the second integral follows
\bee
\int \varphi_R^2   \varphi_L^{2\sigma} dx \le C \hbar^{-\alpha}
\eee
for some $\alpha >0$. \ As for the derivatives of $C_j$, the analyticity in $\mu$ of $(H_0- \lambda- \omega\mu )^{-1}$ ensures their regularity, and the uniform boundness follows from (\ref{Eq65}).

We come back to the problem (\ref{Eq66}). \ This problem is mapped to the problem to find the roots of the following characteristic equation. 
\bee 
D(a_R, a_L, E, \mu, \nu)= \mathrm{det}(N+\mu I-C)=0. 
\eee
Concretely,
\bee
\mathrm{det}(N+\mu I-C) 
&=& (\alpha+\mu-\nu C_1)(\beta+\mu-\nu C_4)-(1-\nu C_2)(1-\nu C_3)\\
&=& \mu^2+\{(\alpha+\beta) -\nu (C_1+C_4)\}\mu +\alpha\beta -1\\
&& \hspace{7mm}-\nu (C_2+C_3+\alpha C_4+\beta C_1) +\nu^2 (C_1C_4-C_2C_3).
\eee
Putting $S(\mu, \nu)=-(C_1+C_4)\mu-(C_2+C_3+\alpha C_4+\beta C_1) +\nu (C_1C_4-C_2C_3)$, we have 
\bee
D(\mu, \nu)=P(\mu)-\nu S(\mu, \nu)=0.
\eee
We note that by the above arguments, $S(\mu, \nu)$ and $\partial_{\mu} P(\mu)$ is uniformly bounded on $D \times K_{\hbar_0}$ for any 
$\hbar \in (0,\hbar_0)$. \ It is also seen that $D(\mu,\nu)$ is a $C^1$ function in $(\mu,\nu)$, 
\bee
&& D(\mu_1, 0)=D(\mu_2,0)=0, \quad \\
&&\frac{\partial D(\mu_1, 0)}{\partial \mu}=2 \mu_1+\alpha+\beta \ne 0, \quad 
\frac{\partial D(\mu_2, 0)}{\partial \mu}=2 \mu_2 +\alpha+\beta \ne 0.
\eee
By applying Implicit Function Theorem, there exist $\varepsilon_0>0$ such that  there exist two real solutions $\mu_1(\nu)$ and $\mu_2(\nu)$ of $D(\mu,\nu)=0$ for $|\nu| <\varepsilon_0$ and that  
\be \label{Eq68}
\mu_1(\nu) &=& \mu_1 -\nu \frac{S(\mu_1, 0)}{\partial_\mu P(\mu_1)}+O(\nu^2), \\ 
\label{Eq69}  
\mu_2(\nu) &=& \mu_2 -\nu \frac{S(\mu_2, 0)}{\partial_\mu P(\mu_2)}+O(\nu^2). \\ 
\ee
Therefore, for any $\varepsilon>0$ there exists $\hbar_1 \in (0, \hbar_0)$ such that $|\mu_1(\nu)-\mu_1|< \varepsilon$, and that $|\mu_2(\nu)-\mu_2|<\varepsilon$ and $\mu_1(\nu), \mu_2(\nu) \in K_{\hbar_0}$ for any $\hbar \in (0,\hbar_1)$. \ We remark here that $L_+^{\lambda,\epsilon}$ has at least one negative eigenvalue since $\langle L_+^{\lambda,\epsilon} \phi_{\epsilon, \lambda}, \phi_{\epsilon, \lambda} \rangle <0$. \ As a consequence, for the symmetric solutions, $L_+^{\lambda,\epsilon}$ has two negative eigenvalues if $|\eta| >\eta^\star$ and has only one negative eigenvalue if $|\eta| <\eta^\star$. \ For the asymmetric solution, $L_+^{\lambda,\epsilon}$ has only one negative eigenvalue. \ The proof of Lemma \ref {Lemma5} has been completed. 
\end {proof}

\subsection {$L^{\lambda , \epsilon}_-$ is a non-negative operator}

Next our aim is proving that $L_-^{\lambda,\epsilon}$ has no negative eigenvalues. Since the symmetric solution $\psi_E^s$, i.e. $\phi_{ \epsilon,\lambda}^s$, is positive by means of a suitable choice of the phase, $L_-^{\lambda,\epsilon} [\phi_{ \epsilon, \lambda}^s]$ is non-negative. \ However, we do not know the sign of the asymmetric solutions and we repeat here the same argument as in Lemma \ref{Lemma5} for $L_-^{\lambda,\epsilon}$.  

\begin {lemma} \label {Lemma6}
%Assume that $\sigma$ satisfies (\ref{Eq21}). 
Let $\phi_{\lambda , \epsilon }$ be the symmetric and asymmetrical stationary solution associated to the level $\lambda$. \ Then there exists $\hbar_2 \in (0,\hbar_1)$, where $\hbar_1$ has been defined in Lemma \ref {Lemma5}, such that for any $\hbar \in (0,\hbar_2)$, $L_-^{\lambda,\epsilon}$ has no negative eigenvalues, more precisely, $L_-^{\lambda,\epsilon}$ has a zero eigenvalue and one positive eigenvalue $\omega \mu = O(\hbar^2)$.
\end {lemma}

\begin {proof}
The eigenvalue problem $L_-^{\lambda,\epsilon} u = (\omega \mu ) u$ with $u \in H^2(\R^d)$, where $|\omega \mu |\le C \hbar^2$, takes the form
\bee
\left \{
\begin {array}{lcl}
\omega \mu a_R &=& a_R \Omega - a_L \omega - \lambda a_R - a_R \langle \varphi_R , g |\phi_{\lambda , \epsilon }|^{2\sigma} u \rangle \\ 
\omega \mu a_L &=& a_L \Omega - a_R \omega - \lambda a_L - a_L \langle \varphi_L , g |\phi_{\lambda , \epsilon }|^{2\sigma} u \rangle \\ 
\omega \mu u_c &=& (H_0 - \lambda ) u_c 
- \Pi_c g |\phi_{\lambda , \epsilon }|^{2\sigma} u 
\end {array}
\right. 
\eee
where we put $u=a_R\varphi_R +a_L \varphi_L +u_c, \quad u_c \in \Pi_c L^2$. \ We remind that $\mu =0$ is a solution of the eigenvalue problem since $L_-^{\lambda , \epsilon} \phi_{\lambda , \epsilon } =0$, we then apply again the same Lyapunov-Schmidt reduction as in Lemma \ref {Lemma5} in order to compute the other eigenvalues of $L_-^{\lambda , \epsilon}$ such that $|\mu \omega | \le C \hbar^2$. \ This eigenvalue problem can be rewritten, assuming $|\omega \mu | \le C \hbar^2$, as follows.  
\bee
(N'+\mu I - \nu C') 
\begin{pmatrix}
a_R \\
a_L
\end{pmatrix}
=
\begin{pmatrix}
0\\
0 
\end{pmatrix}
,
\eee
where 
\bee
&& N'=
\begin{pmatrix}
\alpha ', & 1 \\
1, & \beta ' 
\end{pmatrix}
, \quad C'=
\begin{pmatrix}
C'_1, & C'_2 \\
C'_3, & C'_4 
\end{pmatrix}
, \quad C_3'=\bar{C_2'},\\
&& \alpha '= E + |\eta | |a_R^\lambda |^{2\sigma}, \quad 
\beta '= E + |\eta | |a_L^\lambda |^{2\sigma}.   
\eee
Remind that $\nu$ is defined in Lemma \ref{Lemma5}. \ As in the proof of Lemma \ref{Lemma5}, it suffices to know the sign of $\alpha '\beta '-1$. \ We compute the case of the asymmetric solutions corresponding to $z^{\lambda} \ne 0$ (in the case of the symmetric solution corresponding to $z^\lambda =0$ we follow the same arguments). \ In this case, 
\bee
\alpha '&\sim & -\sqrt{1-(z^{\lambda})^2}-|\eta|\{(p^{\lambda})^{2\sigma+2} + (q^{\lambda})^{2\sigma+2}\} + |\eta| (p^{\lambda})^{2\sigma}, \\
\beta '&\sim & -\sqrt{1-(z^{\lambda})^2}-|\eta|\{(p^{\lambda})^{2\sigma+2} + (q^{\lambda})^{2\sigma+2}\} + |\eta| (q^{\lambda})^{2\sigma}, \\
&& |\eta|=\Big[\Big(\frac{1+z^{\lambda}}{2}\Big)^{\sigma} -\Big(\frac{1-z^{\lambda}}{2}\Big)^{\sigma}\Big]^{-1} 
\times \frac{2z^{\lambda}}{\sqrt{1-(z^{\lambda})^2}}.
\eee
By direct computations it is not difficult to obtain that  
\bee
\alpha '\sim \frac{z^{\lambda}-1}{\sqrt{1-(z^{\lambda})^2}}, \quad \beta '\sim  -\frac{z^{\lambda}+1}{\sqrt{1-(z^{\lambda})^2}}.
\eee
Therefore, $\alpha ' \beta ' \sim 1$ and 
\bee
\alpha '+\beta '\sim -\frac{2}{\sqrt{1-(z^{\lambda})^2}},
\eee
which implies $\mu_1 \mu_2 \sim 0$ and $\mu_1 +\mu_2 >0$ for the eigenvalues of $N'$. \ We may assume without generality that 
$|\mu_1|$ is very small and $\mu_2$ is positive. \ It follows from the same arguments as in Lemma \ref{Lemma5} that the perturbed matrix $N'-\nu C'$ has two different eigenvalues $\mu_1(\nu)$ and $\mu_2(\nu)$ verifying (\ref{Eq68}) and  (\ref{Eq69}). \ Since we know that $L_-^{\lambda,\epsilon}$ has always zero eigenvalue, and perturbed eigenvalues are continuous with respect to $\nu$, we conclude that  $\mu_1(\nu)=0$ and $\mu_2(\nu)>0$. 
\end {proof}

\subsection {Slope condition}

In order to check the slope condition, we consider the following quantity.  
\bee
F(\lambda ) =  \| \phi_{\lambda ,\epsilon } \|^2 = |\epsilon |^{1/\sigma}\,  
\eee
and we remark that
\bee
\frac {d F (\lambda)}{d \lambda } &=& \left [ \frac {d \lambda }{d\epsilon } \right ]^{-1} \frac {d}{d\epsilon} (-\epsilon)^{1/\sigma} = - \frac {1}{\omega \sigma } |\epsilon |^{(1-\sigma)/\sigma} \left [ \frac {d E }{d\epsilon } \right ]^{-1} \\ 
&=& - \frac {1}{C_R \sigma } |\epsilon |^{(1-\sigma)/\sigma} \left [ \frac {d E }{d\eta } \right ]^{-1}\, . 
\eee
Thus, we only have to check the sign of $\frac {d E }{d\eta }$ for the symmetric and asymmetrical stationary solutions.

\subsubsection {Estimate of the stationary solutions as function of the non-linearity parameter.}

The stationary solution 
\bee
\psi = a_R \varphi_R + a_L \varphi_L + \psi_c \, , 
\eee
of equation (\ref {Eq25}) associated to the energy level $E$ depends on the value of the nonlinearity parameter $\eta = \epsilon c /\omega$, where $c=C_R=C_L$ is defined in equation (\ref {Eq30}). 

In particular, in Theorem 1 we have proved that, locally, there is a correspondence one-to-one from $\eta$ to the solution $p$, $q$, $\alpha$, $\beta$ and $E$ (up to the gauge choice of the phase, where we set $\theta = \alpha - \beta$) of equation (\ref {Eq45}) and $\psi_c$ of equation (\ref {Eq37}); provided that $\eta \not= \pm \eta^+$ and $\eta \not= \pm \eta^\star$.

In order to see the sign of $dE/d\eta,$ we wish to obtain the estimate of the first derivative of $p$, $q$, $\alpha$, $\beta$, $E$ and $\psi_c$ as function of $\eta$. \ To this end, let 
\bee
D'' = \{ ( p,q, \alpha, \beta, E ) \in [0,1]^2 \times [0, 2 \pi )^2 \times \R \ :\ p^2+q^2 \le 1 ,\ |\omega E | \le C \hbar^2 \}
\eee
for some $C>0$ fixed; and let 
\bee
\begin {array}{ccccc}
\Phi & : & \R \times H^2 \times D'' & \to & H^2 \times \R^4 \\
&& (\eta , \psi_c , p, q, \alpha , \beta , E ) & \mapsto & \left ( F(\psi_c ) , G \right )
\end {array}
\eee
where $F(\psi_c)$ is defined by (\ref {Eq38Ter}) and where $G$ is defined by (\ref {Eq45}) 
with $\epsilon$ replaced by $\omega \eta /c$; $F(\psi_c)=F(\eta, \psi_c,p,q,\alpha,\beta,E)$, 
and $G=G(\eta, p,q,\alpha,\beta, E)=(G_1, G_2, G_3, G_4).$ \ For simplicity, we set $y=(\psi_c , p,q, \alpha, \beta, E ) \in H^2 \times D'' $. \ Since the mapping $\displaystyle{\frac{\partial \Phi}{\partial y} (\eta, \cdot)} : H^2 \times D'' \to H^2 \times \R^2$ is one-to-one at a point $\eta \ne \pm\eta^{\star}, \pm\eta^{+}$, 
we obtain the unique solution $y=y(\eta )$ of equation $\Phi (\eta , y ) =0$ (up to the gauge choice of the phase). Furthermore, $\Phi(\eta, y)$ is $C^1$, so the solution $y(\eta)$ is $C^1$ except for $\eta \ne \pm\eta^{\star}, \pm\eta^{+}$, and we have  
\be
\frac {\partial \Phi}{\partial \eta} + \frac {\partial \Phi}{\partial y} y' =0 \, . \label {Eq76Ter}
\ee
Here, $'$ denotes the derivative with respect to $\eta$, and we use this notation hereafter, too. \ 
We will in fact see that $\displaystyle{\frac{\partial \Phi}{\partial y} (\eta, \cdot)}$ is one-to-one 
for any $\eta\ne \pm\eta^{\star}, \pm\eta^{+} $ in the proof of Lemma \ref{Lemma8} below. \ Therefore we do not 
mention the details about this fact here. 

The first equation of (\ref {Eq76Ter}) takes the form 
\be
\omega E \psi_c' + \omega E' \psi_c = [H_0 -\Omega  ] \psi_c' + \frac {\omega}{c}\Pi_c v + \frac{\omega \eta}{c}  \Pi_c g W  \left ( a_R' \varphi_R + a_L' \varphi_L + \psi_c' \right ) \, , \label {Eq76Bis}
\ee
where $v=g|\psi|^{2\sigma}\psi,$ and 
\bee
W=[(\sigma+1)|\psi|^{2\sigma}+\sigma\psi^2|\psi|^{2(\sigma-1)} \mathcal{T}], 
\quad \mathcal{T} u := \bar{u} \, ; 
\eee
actually, the stationary solution is a real valued function by means of a gauge choice (see Remark 7). 

In order to write the other equations of (\ref {Eq76Ter}) we make use of the first two equations of (\ref {Eq28}) and of the normalization condition:
\be
\left \{
\begin {array}{l}
E  a_R = - a_L +  \frac {\eta}{c} \langle \varphi_R , v \rangle \\ 
E  a_L = - a_R +  \frac {\eta}{c} \langle \varphi_L , v \rangle \\ 
|a_R|^2 + |a_L|^2 + \langle \psi_c , \psi_c \rangle = 1
\end {array}
\right. \, . \label {Eq75}
\ee

Now, we get the estimate of the derivative of $\psi_c$ in Lemma \ref {Lemma7} and then the estimate of the derivative of $p,q,\alpha, \beta$ and $E$ in Lemma \ref {Lemma8}.

\begin {lemma} \label {Lemma7}
Let $(a_R , a_L, E) \in D$ and let $\eta $ satisfying Hyp. \ref {Hyp3}, let $\psi_c$ be the solution of equation (\ref {Eq37}). \ Then 
\be
\left \| \frac {\partial \psi_c}{\partial \eta} \right \|_{H^2} = \left [ 1 + \max \left ( \left | \frac {\partial a_R}{\partial \eta} \right | , \ \left | \frac {\partial a_L}{\partial \eta} \right | , \ \left | \frac {\partial E}{\partial \eta} \right | \right ) \right ] \tilde \asy \left ( e^{-\rho /\hbar} \right ) \, \mbox { as } \ \hbar \to 0 . \label {Eq73}
\ee
\end {lemma}

\begin {proof}
Since equation (\ref {Eq76Bis}) can be written as
\bee
\left [ \left ( H_0 - \Omega - \omega E \right ) \Pi_c  + \epsilon \Pi_c W \right ] \psi_c' = \omega E' \psi_c - \epsilon   \Pi_c g W  \left ( a_R' \varphi_R + a_L' \varphi_L  \right ) + \frac {\omega}{C_R} \Pi_c v \, . 
\eee
then
\bee
\psi_c' &=& \left [ I + \left ( H_0 - \Omega - \omega E \right )^{-1} + \epsilon \Pi_c W \right ] \, \left [  H_0 - \Omega - \omega E \right ]^{-1} \, \times \\ 
&& \ \ \times \left [ \omega {E}' \psi_c - \epsilon   \Pi_c g W  \left ( a_R' \varphi_R + a_L' \varphi_L  \right ) + \frac {\omega}{C_R} \Pi_c v  \right ]
\eee
and, by making use of the same ideas applied in the proof of Lemma \ref {Lemma5}, it turn out that the inverse operator is bounded and (\ref {Eq73}) follows.
\end {proof}

\begin {lemma} \label{Lemma8}
Let $|\eta |\not = \eta^\star $ and $|\eta |\not= \eta^+$. \ Then 
\be
\max \left [ \left | \frac {\partial p}{\partial \eta} \right | , \ \left | \frac {\partial q}{\partial \eta} \right | , \ \left | \frac {\partial \alpha}{\partial \eta} \right | , \ \left | \frac {\partial \beta}{\partial \eta} \right |, \ \left | \frac {\partial E}{\partial \eta} \right | \right ] \le C \label {Eq74}
\ee
for some $C>0$.
\end {lemma}

\begin {proof} Now, in order to give an estimate of the derivative of $p$, $q$, $\alpha$, $\beta$ and $E$ we write down the corresponding equations of (\ref {Eq76Ter}), that is we have to consider the derivate of equations (\ref {Eq75}). \ We assume, for the sake of definiteness, that the stationary solution corresponds to $\theta =0$ (that is $\psi$ is a symmetric or asymmetrical stationary solution). \ In fact, we rewrite $a_R=p e^{i\alpha}$ and $a_L = q e^{i\beta}$ by means of $p$, $q$, $\alpha$ and $\beta$ (where we set $\theta = \alpha - \beta$); so that equation (\ref {Eq75}) takes the form
\bee
\left \{
\begin {array}{lcl}
E p + q \cos \theta - \frac {\eta}{C_R} \Re \left [  \langle \varphi_R , v \rangle e^{-i \alpha} \right ] &=& 0 \\ 
 && \\
q \sin \theta + \frac {\eta}{C_R} \Im \left [  \langle \varphi_R , v \rangle e^{-i \alpha} \right ] &=& 0 \\
 && \\
E q + p \cos \theta - \frac {\eta}{C_R} \Re \left [ \langle \varphi_L , v \rangle e^{-i \beta} \right ] &=& 0 \\
 && \\ 
- p \sin \theta + \frac {\eta}{C_R} \Im \left [ \langle \varphi_L , v \rangle e^{-i \beta} \right ] &=& 0 \\ 
&& \\ 
p^2 + q^2 + \| \psi_c \|^2 &=& 1 
\end {array}
\right.
\eee
We take now the derivative of both sides with respect to $\eta$, obtaining that
\bee
\left \{
\begin {array}{l}
E' p + E p' +q' \cos \theta - q \theta' \sin \theta  - \frac {\eta}{C_R} \Re \left [  \langle \varphi_R , v' \rangle e^{-i \alpha} -i \alpha' \langle \varphi_R , v \rangle e^{-i \alpha} \right ] =  \frac {1}{C_R} \Re \left [  \langle \varphi_R , v \rangle e^{-i \alpha} \right ] \\ 
 \\
q' \sin \theta + q \theta' \cos \theta + \frac {\eta}{C_R} \Im \left [  \langle \varphi_R , v' \rangle e^{-i \alpha} -i \alpha'  \langle \varphi_R , v \rangle e^{-i \alpha} \right ] =  - \frac {1}{C_R} \Im \left [  \langle \varphi_R , v \rangle e^{-i \alpha} \right ] \\ 
 \\
E' q + E q' +p' \cos \theta - p \theta' \sin \theta  - \frac {\eta}{C_R} \Re \left [ \langle \varphi_L , v' \rangle e^{-i \beta} - i \beta' \langle \varphi_L , v \rangle e^{-i \beta}\right ] = \frac {1}{C_R} \Re \left [ \langle \varphi_L , v \rangle e^{-i \beta} \right ] \\
 \\
- p' \sin \theta - p \theta' \cos \theta   + \frac {\eta}{C_R} \Im \left [ \langle \varphi_L , v' \rangle e^{-i \beta} - i \beta' \langle \varphi_L , v \rangle e^{-i \beta} \right ] = - \frac {1}{C_R} \Im \left [ \langle \varphi_L , v \rangle e^{-i \beta} \right ] \\ 
 \\ 
2p p' + 2 q q' = - 2\Re \langle \psi_c , \psi_c' \rangle 
\end {array}
\right.
\eee
We remark that
\bee
\begin {array}{c}
\langle \psi_c, \psi_c' \rangle  = \asy (\nu^2 ) \\ 
\langle \varphi_R, v \rangle = \langle \varphi_R , g |\varphi_R |^{2\sigma} \varphi_R \rangle |a_R|^{2\sigma } a_R + \tilde \asy (\nu) = C_R p^{2\sigma+1} e^{i\alpha}  + \tilde \asy (\nu) \\ 
\langle \varphi_L, v \rangle = \langle \varphi_L , g |\varphi_L |^{2\sigma} \varphi_L \rangle |a_L|^{2\sigma } a_L + \tilde \asy (\nu) = C_L q^{2\sigma+1} e^{i\beta} + \tilde \asy (\nu) \\
\langle \varphi_R , v' \rangle = \langle \varphi_R , g W \varphi_R \rangle (p' + p i \alpha ') e^{i\alpha } + q' \tilde \asy (\nu ) + \beta' \tilde \asy (\nu ) +  \tilde \asy (\nu ) \\ 
\langle \varphi_L , v' \rangle = \langle \varphi_L , g W \varphi_L \rangle (q' + q i \beta ') e^{i\beta }  + p' \tilde \asy (\nu ) + \alpha' \tilde \asy (\nu ) + \tilde \asy (\nu ) 
\end {array}
\eee
where
\bee
\begin {array}{c}
\nu = e^{-\gamma \rho /\hbar} \\
\langle \varphi_R, g W \varphi_R \rangle = C_R \left [ (\sigma +1)  + \sigma e^{i 2 \alpha} \right ]  p^{2\sigma}+ \tilde \asy (\nu ) \\ 
\langle \varphi_L, g W \varphi_L \rangle = C_L \left [ (\sigma +1)  + \sigma e^{i 2 \beta} \right ] q^{2\sigma}+ \tilde \asy (\nu )
\end {array}
\eee
Therefore, the above system takes the form (where the asymptotics $\sim $ means that the remainder term is of order $\tilde \asy (\nu )$)
\bee
\left \{
\begin {array}{l}
E' p + E p' +q' \cos \theta - q \theta' \sin \theta  + \\ 
 \ \ - \frac {\eta}{C_R} \Re \left [  C_R [(\sigma +1) + \sigma e^{2i\alpha} ] p^{2\sigma} (p'+ip \alpha') -i \alpha' C_R p^{2\sigma +1} \right ]  \sim  p^{2\sigma+1} \\ 
 \\
q' \sin \theta + q \theta' \cos \theta + \\ 
\ \ + \frac {\eta}{C_R} \Im \left [  C_R [(\sigma +1) + \sigma e^{2i\alpha} ] p^{2\sigma} (p'+ip \alpha') -i \alpha' C_R p^{2\sigma +1} \right ] \sim  0 \\ 
 \\
E' q + E q' +p' \cos \theta - p \theta' \sin \theta  + \\ 
\ \ - \frac {\eta}{C_R} \Re \left [  C_R [(\sigma +1) + \sigma e^{2i\beta} ] q^{2\sigma} (q'+iq \beta') -i \beta' C_R q^{2\sigma +1} \right ]  \sim  q^{2\sigma+1}  \\
 \\
- p' \sin \theta - p \theta' \cos \theta   + \\
\ \ +  \frac {\eta}{C_R} \Im \left [  C_R [(\sigma +1) + \sigma e^{2i\beta} ] q^{2\sigma} (q'+iq \beta') -i \beta' C_R q^{2\sigma +1} \right ] \sim  0\\ 
 \\ 
2p p' + 2 q q'  \sim  0 
\end {array}
\right.
\eee
that is
\bee
M (1 + \tilde \asy (\nu )) 
\left (
 \begin {array} {l}
 E' \\ p' \\ q' \\ \alpha' \\ \beta' 
 \end {array} 
 \right ) 
 = 
\left (
\begin {array}{l}
p^{2\sigma +1}  \\ 
0 \\
q^{2\sigma +1}  \\  
0 \\ 
0 
\end {array} 
\right ) 
\eee
where
\bee
M = 
\left ( 
\begin {array}{ccc}
p & E  - \eta [(\sigma +1) + \sigma \cos (2 \alpha )] p^{2\sigma }  & \cos \theta    \\
0 & \eta \sigma p^{2\sigma} \sin (2\alpha ) & \sin \theta \\ 
q & \cos \theta & E  - \eta [(\sigma +1) + \sigma \cos (2 \beta  )] q^{2\sigma }  \\
0 & -\sin \theta & \eta \sigma q^{2\sigma} \sin (2\beta ) \\ 
0 & 2p & 2q 
\end {array}
\right. \\
\left. \ \ \ \   
\begin {array}{cc}
 -q \sin \theta + \eta \sigma p^{2\sigma+1} \sin (2\alpha ) & + q \sin \theta  \\
 q \cos \theta + \eta p^{2\sigma +1} \sigma [1+ \cos (2\alpha ) ] & - q \cos \theta  \\ 
 - p \sin \theta & +p \sin \theta + \eta \sigma q^{2\sigma+1} \sin (2\beta ) \\
 - p \cos \theta  & p \cos \theta + \eta q^{2\sigma +1} \sigma [1+ \cos (2\beta ) ] \\ 
 0 & 0
\end {array}
\right )
\eee
We consider now, separately, the symmetric and asymmetrical solutions.

{\it Symmetric solution.} In the case of the symmetric solution where $\theta =0$ we can choose the common phase $\alpha = \beta =0$, by means of a gauge choice. \ Since $p=q= \frac {1}{\sqrt {2}}$, then a straightforward calculation gives that the matrix $M$ takes the form
\bee
\mbox {det} (M) = -8 \sigma \eta 2^{-\sigma} \left ( 1 + \eta \sigma 2^{-\sigma } \right )
\eee
Hence, for $\eta <0$ then det $(M)\not= 0$ provided that $|\eta | \not= \eta^\star.$ \ Hence, we have that (\ref {Eq74}) holds true. 

{\it Asymmetrical solution.} In the case of the asymmetrical solution corresponding to $\eta <0$ then $\theta =0$, we can still choose the common phase $\alpha = \beta =0$ by means of a gauge choice, and $p= \sqrt {\frac {1+z}{2}}$ and $q= \sqrt {\frac {1-z}{2}}$ satisfy equation $f_+ (z,\eta)=0$. \ Then we can set
\bee
\eta = - \frac {2z}{\sqrt{1-z^2}} \left [ \left ( \frac {1+z}{2} \right )^\sigma - \left ( \frac {1-z}{2} \right )^\sigma \right ]^{-1}
\eee
and
\bee
E \sim - \sqrt {1-z^2} + \eta \left [ \left ( \frac {1+z}{2} \right )^{\sigma +1} + \left ( \frac {1-z}{2} \right )^{\sigma +1} \right ]
\eee
By means of a straightforward computation it turns out that
\bee
\mbox {det M} = 8 \sigma z 2^{\sigma+1} \frac {[-(4\sigma +2)z (1-z^2)^\sigma +(1+z)^{2\sigma +1} - (1-z)^{2\sigma +1}] [h(z)]}{(1-z^2) [(1+z)^\sigma - (1-z)^\sigma]^3}
\eee
where $h(z) =g(z)-g(-z) $ enters in the definition of $\eta' $ (see equation (51)). \ If we remark that the function 
\be
Q(z):=[-(4\sigma +2)z (1-z^2)^\sigma +(1+z)^{2\sigma +1} - (1-z)^{2\sigma +1}] \label {Eq76}
\ee
is such that $Q(0)=0$ and that
\bee
\frac {dQ}{dz}= (2\sigma +1) \left [ 4 z^2 \sigma (1-z^2)^{\sigma -1} + \left ( (1-z)^\sigma - (1+z)^\sigma \right )^2 \right ] >0 \, , \ \forall z \in [-1,+1], \ z \not= 0 ,
\eee
then we can conclude that $\mbox {det M}=0$ if, and only if, $z=0$ and $z$ is a zero of the function $h(z)$. \ Then, as in the case of symmetric solution then (\ref {Eq74}) holds true. \ The Lemma is so proved. \end {proof}

\begin {remark} \label {Remark17}
In fact, for symmetric solution a straightforward calculation gives that
\be
 \left (
 \begin {array} {l}
 E' \\ p' \\ q' \\ \alpha'  \\ \beta'
 \end {array} 
 \right ) 
 \sim
M^{-1} 
\left (
\begin {array}{l}
p^{2\sigma +1} \\
0 \\
q^{2\sigma +1} \\ 
0  \\  
0 
\end {array} 
\right ) 
=
\left (
\begin {array}{l}
2^{-\sigma} \\ 
0  \\  
0  \\ 
0 \\ 
0 
\end {array} 
\right )  \label {Eq77}
\ee
On the other hand, for asymmetrical solution corresponding to $z=z^{as}$ a straightforward calculation gives also that
\be
 \left (
 \begin {array} {l}
 E' \\ p' \\ q' \\ \alpha'  \\ \beta'
 \end {array} 
 \right ) 
 \sim
M^{-1} 
\left (
\begin {array}{l}
p^{2\sigma +1} \\
0 \\
q^{2\sigma +1} \\ 
0  \\  
0 
\end {array} 
\right ) 
=
\left (
\begin {array}{l}
\frac {Q(z)}{2^{\sigma +1} h(z)}  \\ 
- \frac {\sqrt {2} \left [ (1+z)^\sigma - (1-z)^\sigma \right ]^2 (1-z^2) \sqrt {1-z}}{2^{\sigma +3}h(z) }  \\  
\frac {\sqrt {2} \left [ (1+z)^\sigma - (1-z)^\sigma \right ]^2 (1-z^2) \sqrt {1+z}}{2^{\sigma +3}h(z) }  \\  
0 \\ 
0 
\end {array} 
\right )  \label {Eq78}
\ee
where the function $Q(z)$, defined in equation (\ref {Eq76}), is such that $Q(-z)=-Q(z)$, $Q(0)=0$ and $\frac {dQ}{dz} > 0 $ for any $z \in (0,1]$.
\end {remark}

Now, we are ready to go back to the slope condition and to state the following.

\begin {lemma} \label{Lemma9} 
%Assume that $\sigma$ satisfies (\ref{Eq21}).
There exists $\hbar_3 \in (0, \hbar_2)$ such that 
for any $\hbar \in (0, \hbar_3)$ the following statements are satisfied. \ Let
\bee
F_s (\lambda ) = \| \phi^s_{\lambda , \epsilon} \|^2 
\eee
where $\phi^s_{\lambda , \epsilon }$ is the symmetric stationary solutions. \ Then 
\bee
\frac{d}{d\lambda} F_s (\lambda ) <0. 
\eee
Moreover, 

\begin {itemize}

\item [(i)] Let $\sigma \le \sigma_{threshold}$ and let 
\bee
F_{as} (\lambda ) = \| \phi^{as}_{\lambda , \epsilon} \|^2
\eee
where $\psi^{as}_{\lambda , \epsilon }$ is the asymmetrical stationary solutions. \ Then 
\bee
\frac{d}{d\lambda} F_{as} (\lambda )<0.
\eee

\item [(ii)] Let $\sigma > \sigma_{threshold}$ and let 
\bee
F_{as1} (\lambda ) = \| \phi^{as1}_{\lambda , \epsilon} \|^2 \ \mbox { and } \ F_{as2} (\lambda ) 
= \| \phi^{as2}_{\lambda , \epsilon} \|^2
\eee
where $\psi^{as1}_{\lambda , \epsilon }$ and $\psi^{as2}_{\lambda , \epsilon }$ are the asymmetric stationary solutions. \ Then 
\bee
\frac{d}{d\lambda} F_{as1} (\lambda ) <0 \ \mbox { and } \ \frac{d}{d\lambda} F_{as2} (\lambda ) >0.
\eee
\end {itemize}

\end {lemma}

\begin {proof} We consider, at first, the case of the symmetric stationary solution corresponding to $z^\lambda = z^s = 0$. \ In such a case from (\ref {Eq77}) it follows that $\frac {d E }{d\eta } = 2^{-\sigma} >0$ and thus $\frac {d F_s (\lambda)}{d \lambda } <0 $ proving so the first statement.

Now, we consider the case of asymmetrical stationary solution corresponding to $z^\lambda \not= 0$. \ In such a case from (\ref {Eq78}) it follows that
\bee
\frac {d E }{d\eta } = \frac {Q(z^\lambda)}{2^{\sigma +1} h(z^\lambda)}
\eee
is an even function and where $Q(z^\lambda) \cdot z^\lambda >0$. \ Hence, the sign of $\frac {d E }{d\eta }$ only depends on the sign of $h (z^\lambda)$. \ We have then showed all the statements in Lemma \ref {Lemma9}, recalling that (see the results in Section \ref {Sec3})

\begin {itemize}

\item [] If $\sigma \le \sigma_{threshold}$, then the asymmetrical stationary solution $\phi_{\lambda , \epsilon}^{as}$ corresponding to $z^\lambda >0$ satisfies condition $h (z^\lambda )>0$;

\item [] If $\sigma > \sigma_{threshold}$, then the asymmetrical stationary solution $\phi_{\lambda , \epsilon}^{as1}$ corresponding to $z^\lambda >0$ satisfies condition $h (z^\lambda )>0$;

\item [] If $\sigma > \sigma_{threshold}$, then the asymmetrical stationary solution $\phi_{\lambda , \epsilon}^{as2}$ corresponding to $z^\lambda >0$ satisfies condition $h (z^\lambda )<0$.

\end {itemize}
\end {proof}

\begin {remark} 
In the same way, the monotone decreasing behavior of
\bee
F_a (\lambda ) = \| \phi_{\lambda ,\epsilon }^a \|^2 \eee
associated to the antisymmetric stationary solution follows.
\end {remark}

Finally, collecting the results of Proposition \ref {Proposition1} and of Lemmata \ref {Lemma5}, \ref {Lemma6} and \ref {Lemma9} then Theorem \ref {Theorem3} follows.

\begin{remark}
In Theorem \ref{Theorem3}, 
in case of $\sigma>\sigma_{threshold}$ and $|\eta|=\eta^+$, we did not obtain any conclusion about the orbital stability. \ Recall that $\eta^+ \in (0,\infty)$ is defined by $\eta^+=|\eta(z^+)|$ with $z^+ \in (0,1)$ 
such that $\eta'(z^+)=0$ (see Theorem \ref{Theorem1}).  \ Let $\phi_{\lambda^+, \epsilon^+}$ be the corresponding asymmetric  stationary solution to $\lambda^+=\Omega+\omega E^+$ where 
\begin{equation*}
E^+ \sim -\sqrt{1-(z^+)^2}+\eta^+ 
\Big[\Big(\frac{1+z^+}{2}\Big)^{\sigma+1}+\Big(\frac{1-z^+}{2}\Big)^{\sigma+1}\Big], 
\end{equation*}
and $\epsilon^+$ is given by $\omega \eta^+ /c$. \ According to Remark \ref{Remark17}, 
we see {\it formally} 
\begin{equation}
\frac{dF(\lambda)}{d\lambda} \Big|_{\lambda=\lambda^+} 
=- \frac{1}{C_R\sigma}|\epsilon|^{(1-\sigma)/\sigma}
\Big(\frac{dE}{d\eta}\Big)^{-1} \Big|_{|\eta|=\eta^+}=0,  
\end{equation}
since $\eta'(z^+)=0.$ 
Thus, we are required to prove the stability/instability 
for the case $dF/d\lambda=0$. 
In fact, this case would be included in (2) 
of Proposition \ref{Proposition1}, 
and we would conclude that, when $\sigma>\sigma_{threshold}$, 
at the transition point $|\eta|=\eta^{+}$ 
from $\phi_{\lambda,\epsilon}^{as1}$ 
to $\phi_{\lambda,\epsilon}^{as2}$, we should have the instability.  
To show this fact exactly, it suffices to compute $\frac{d^2 F}{d\lambda^2}$ 
and prove that it is not zero at $\lambda=\lambda^+,$ following the argument 
in Maeda \cite{Maeda} (see also some related conditions in \cite{CoPe, Ohta}). 
At least ``formally'' this may be seen as follows: 
we note that the use of Budan-Fourier 
Theorem ensures $d\lambda(z)/dz$ $\sim$ negative for $\hbar$ small. 
By formal calculations, 
\begin{eqnarray*}
\frac{d^2 F}{d \lambda^2}&=&\Big(\frac{\omega}{C_R}\Big)^{1/\sigma} 
\Big\{\frac{1}{\sigma}\Big(\frac{1}{\sigma}-1\Big)|\eta|^{\frac{1}{\sigma}-2}
\Big(\frac{d\eta}{d\lambda}\Big)^2 
-\frac{1}{\sigma}|\eta|^{\frac{1}{\sigma}-1}\frac{d^2 \eta}{d\lambda^2}\Big\}, \\ 
\frac{d\eta}{d\lambda}&=&\eta'(z)/\frac{d\lambda}{dz}, \quad 
\frac{d^2\eta}{d\lambda^2}=\Big\{\eta''(z) \frac{d\lambda}{dz}-\eta'(z) \frac{d^2 \lambda}{dz^2}\Big\}/\Big(\frac{d\lambda}{dz}\Big)^3. 
\end{eqnarray*}
We have seen in Section 3 that $\eta''(z^+) \ne 0$, which implies 
$\frac{d^2 F}{d \lambda^2} |_{\lambda=\lambda^+} \ne 0.$ 
However a rigorous justification seems more complex 
and we do not pursue in this direction in the present paper.  
\end{remark}

\appendix 

\section {Stationary states for a non-linear toy model}\label {App}

Here, we introduce, as a toy model, the semiclassical Schr\"odinger equation with two attractive symmetric Dirac's $\delta$ which is partially investigated in \cite{KoSa}.  
\be \label{Eq79}
i \hbar \frac {\partial \psi }{\partial t} 
= H_0 \psi + \epsilon g |\psi |^{2\sigma} \psi, 
\quad  \| \psi (\cdot ,t) \| =1, \quad x\in \R,~ t\in \R, 
\ee
where 
\bee
H_0=- \hbar^2 \frac {d^2}{d x^2} + \beta \delta_{-a} + \beta \delta_{+a}
\eee
for some $a \in \R$ and $\beta<0$. \ Hereafter, for the sake of definiteness, we assume that $g \equiv 1$.

Even though this operator $H_0$ with Dirac measures do not satisfy the assumptions for the potential $V(x)$ in the Introduction, the two-level approximation used in the previous sections is directly applicable to this example. \ In this section, we will give some remarks for the properties of $H_0$, and  the general theory we have used in the previous sections, for example, Cauchy problem and the orbital stability. 
%\ {\bf We point out that, due to the singularity of the Dirac's delta, we should treat the %problem of orbital stability in the Hilbert space $L^2$, instead of $H^1$, making use of some %ideas in \cite {FuOhOz}.} 
We remark that a symmetric-breaking phenomenon for the cubic nonlinear Schr\"odinger equation with double Dirac potential is discussed in \cite{JaWe} too, but not in the semiclassical regime. 

\subsection {Spectrum of the linear operator}

The spectral problem
\bee
\left [ - \hbar^2 \frac {d^2}{d x^2} 
+ \beta \delta_{-a} + \beta \delta_{+a} \right ] \psi = {\mathcal E} \psi 
\eee
for $\beta < 0$ is equivalent to the spectral problem
\bee
H_\alpha \psi = E \psi 
\eee
where we set $E={\mathcal E}/\hbar^2$ and where the linear operator
\bee
H_{\alpha} = - \frac {d^2}{d x^2},\quad 
\mbox{with} \quad \alpha = \beta/\hbar^2, \,
\eee
is self-adjoint on the domain 
\bee
D(H_{\alpha}) = \left 
\{\psi \in H^2 (\R \setminus \{\pm a\}) \cap H^1(\R) :~  
\frac{d\psi}{dx}(\pm a +0)-\frac{d\psi}{dx}(\pm a -0) = \alpha \psi (\pm a) \right \}.
\eee

Let us recall some basic properties of the spectrum of $H_{\alpha}$ (see, e.g., \cite{Albeverio,KoSa} for details.) 

The essential spectrum of $H_{\alpha}$ is purely absolutely continuous and coincides with the positive real axis:
\bee
\sigma_{\mbox {\rm \small ess}} (H_{\alpha})
=\sigma_{\mbox {\rm \small ac}} (H_{\alpha}) =[0,+\infty ) \, .
\eee

The discrete spectrum consists of two eigenvalues, at least, given by means of the Lambert's special 
function $W(x)$ such that $W(x) e^{W(x)} = x $.

If $\alpha  < 0$ the discrete spectrum is not empty, in particular, 

\begin {itemize}

\item [-] if $ a  \le - \frac {1}{\alpha}$, then the discrete spectrum of $H_{\alpha}$ consists of only one eigenvalue $E_1 (a ,\alpha)$ defined as 
\bee
E_1 (a,\alpha ) = - \frac {1}{4a^2} \left [ W \left ( - a \alpha e^{a \alpha } \right ) - a  \alpha \right ]^2 \, ;
\eee

\item [-] if $ a  > -\frac {1}{\alpha}$, then the discrete spectrum of $H_{\alpha}$ consists of two eigenvalues $E_1 (a , \alpha )$ and $E_2 (a, \alpha )$ where 
\bee
E_2 (a,\alpha ) = - \frac {1}{4a^2} \left [ W \left ( + a \alpha e^{a \alpha } \right ) - a  \alpha \right ]^2 \, . 
\eee

\end {itemize}

The two associated eigenvectors take the form:

\begin {itemize}

\item [i)] Let 
\bee
k_1 = \sqrt {E_1} = \frac {i}{2a} \left [ W \left ( - a \alpha e^{a \alpha } \right ) - a \alpha \right ] 
\eee
then
\bee
\varphi_1 ( x) = C_1 
\left \{ 
\begin {array}{ll} 
e^{-i k_1 x} & x < -a \\ 
\frac {2k_1 + i \alpha}{2k_1} \left ( e^{-i k_1 x} + e^{i k_1 x } \right ) & 
-a \le x \le +a \\ 
e^{+i k_1 x } & x > + a 
\end {array}
\right.
\eee
where $C_1$ is the normalization constant given by
\bee
C_1 = \frac {|k_1|}{\sqrt {(2|k_1|+\alpha ) \, ( 2 |k_1| a + a \alpha + 1 )}} \, . 
\eee

\item [ii)] Let
\bee
k_2 = \sqrt {E_2} = \frac {i}{2a} \left [ W \left ( + a \alpha e^{a \alpha } \right ) - a \alpha \right ] 
\eee
then
\bee
\varphi_2 ( x) = C_2 
\left \{ 
\begin {array}{ll} 
e^{-i k_2 x} & x < -a \\ 
\frac {2k_2 + i \alpha}{2k_2} \left ( e^{-i k_2 x} - e^{i k_2 x } \right ) & 
-a \le x \le +a \\ 
- e^{+i k_2 x } & x > + a 
\end {array}
\right.
\eee
where $C_2$ is the normalization constant given by
\bee
C_2 = \frac {|k_2|}{\sqrt {-(2|k_2|+\alpha ) \, ( 2 |k_2| a + a \alpha + 1 )}} \, . 
\eee

\end {itemize}

\begin {remark}
Recalling that the Lambert's special function $W(x)$ has the following asymptotic behavior 
\bee
W(x) \sim x-x^2 + \frac 32 x^3 + \asy (x^4)
\eee
then it follows that the splitting is exponentially small, namely
\bee
\left | {\mathcal E}_1 - {\mathcal E}_2 \right | \sim \hbar^2 \alpha^2 e^{a \alpha } = \frac {\beta^2}{\hbar^2} e^{a\beta/\hbar^2} = \frac {\beta^2}{\hbar^2} e^{- a|\beta | /\hbar^2}.
\eee
\end {remark}
 
\begin{remark} \label{Remark20} 
The resolvent formula for $H_{\alpha}$ is known: let $h \in C_0^{\infty}(\R)$, $k^2 \in \rho(H_{\alpha}),$ and $\Im k >0$. The resolvent is expressed as follows. 
\bee
\left ( [H_\alpha - k^2 ]^{-1} h \right ) (x) = \int_{\R} K_\alpha (x,y;k) h (y) d y,  
\eee
with the kernel $K_\alpha$ having the following form 
\bee
K_\alpha (x,y ; k ) = K_0 (x,y;k) + \sum_{j=1}^4 K_\alpha^j (x,y;k)
\eee
where
\bee
K_0 (x,y;k) &=& \frac {i}{2k} e^{ik|x-y|} \\
K_\alpha^1 (x,y;k) &=& \frac {\alpha (2k +i\alpha )}{2k\left ( (2k+i\alpha )^2 + \alpha^2 e^{i4ka} \right )} e^{ik |x+a| + i k |y+a|} \\ 
K_\alpha^2 (x,y;k) &=& \frac {- i \alpha^2 e^{2i ka}}{2k\left ( (2k+i\alpha )^2 + \alpha^2 e^{i4ka} \right )} e^{ik |x+a| + i k |y+a|} \\ 
K_\alpha^3 (x,y;k) &=&  K_\alpha^2 (-x,-y;k) \\
K_\alpha^4 (x,y;k) &=&  K_\alpha^1 (-x,-y;k) 
\eee
\end {remark}
\vspace{3mm}

We consider here \emph {the case $a>-1/\alpha$ with $\alpha<0$}.  \ In such a case we have that the linear problem has two negative non degenerate eigenvalues:
\be
E_1 < E_2 < 0 \, . \label {Eq80}
\ee

\subsection {Nonlinear problem}

The local existence of solution in $H^1(\R)$, and conservation laws of energy and $L^2$ norm are verified in a similar way to \cite{FuOhOz}; the authors in \cite{FuOhOz} applied Theorem 3.7.1 of \cite{C} to the case of $a=0$. \ In our case, we take $-H_{\alpha}+E_1$ for the operator $A$ of Theorem 3.7.1 of \cite{C}. \ Then this operator A is a self adjoint operator on $X=L^2(\R)$ with the domain $D(A)=D(H_{\alpha})$, and also $A \le 0$. \ We take $X_A=H^1(\R)$ whose norm is equivalent to $H^1(\R)$ norm 
\bee
\|v\|_{X_A}^2=\|(d/dx)v\|^2+(1-E_1)\|v\|^2+\alpha (|v(a)|^2+|v(-a)|^2).
\eee
Condition (3.7.2) of Theorem 3.7.1 of \cite{C} is satisfied with $p=2$, and other conditions hold since we are in one dimensional case.  
\vspace{3mm}

For the existence of bifurcation of stationary solutions, it suffices to repeat the similar arguments in Section 3 (Theorem \ref{Theorem1}), but in $H^1(\R)$ instead of $H^2(\R)$.   
\vspace{3mm}

We can check the assumptions 
for the orbital stability/instability of stationary states in $H^1(\R)$, as in Section \ref{Sec5}, using the two level approximation. \ However, due to the singularity of Dirac potentials, we cannot consider  
the linearized problem with a more smooth domain than $H^1(\R)$, as, for ex., was considered in \cite{DiGa}. \ Remark also that $H^2$ regularity allows us simply to have the nonlinear instability assuming the existence of an unstable eigenvalue (e.g. \cite{CoCoOhIHP}). \ We thus give some explanations here.  

We consider as follows the linearized problem around the real valued rescaled stationary state $\phi_{\epsilon, \lambda}$ ($\epsilon$ and $\lambda$ are fixed here to discuss the general theory, so we denote it simply by $\phi$ from now on). 
\be \label{Eq81}
\frac{dv}{dt}=Av +F(v), \quad v=(v_1, v_2) \in D(A) \quad  
\mbox{with} \quad v_1=\Re v\, , \quad v_2=\Im v \, , 
\ee
where $A(v_1, v_2)=(L_-^{\lambda,\epsilon} v_2, -L_+^{\lambda,\epsilon} v_1)$. \ $A$ is a linear operator in $\mathbb{L}^2(\R)$ with domain 
\bee
D(A)=\Big\{v \in \mathbb{H}^2(\R \setminus \{\pm a\}) \cap \mathbb{H}^1(\R)~:
~ \frac{d v_j}{dx}(\pm a+0)- \frac{d v_j}{dx}(\pm a-0)=\alpha v_j(\pm a),~j=1,2\Big\}, 
\eee
where $\mathbb{H}^m (\R)=H^m (\R) \times H^m(\R)$ for $m\in \mathbb{Z}$. \ The nonlinear term is given by
\bee
F(v)=i\{|\phi+v|^{2\sigma}(\phi+v)-|\phi|^{2\sigma+1}-(\sigma+1)|\phi|^{2\sigma}v
-\sigma |\phi|^{2\sigma} \bar{v}\}.
\eee 

This operator $A$ generates its $C_0$-semigroup on $\mathbb{L}^2$ denoted by $e^{tA}$. \ Concerning the spectrum of $A$, we have the following Lemma. \ We note that we complexify the space when we consider the spectrum problem of $A$.

\begin{lemma} 
$\sigma_{\mathrm{ess}}(A) \subset i\R$. 
\end{lemma}

\begin{proof}
The operator $A$ can be rewritten in the following form (still denoted by $A$ with abuse of notation)
\bee
A v = - i \left \{ H_\alpha - \lambda - (\sigma +1)  |\phi |^{2\sigma } 
- \sigma  |\phi |^{2\sigma} {\mathcal T} \right \} v 
\eee
where ${\mathcal T} v = \bar v$ is a non-symmetric bounded linear operator. \ We consider the operator $iA$ as the operator $A_0$ perturbed by the operator $C$, i.e. 
\begin{equation*}
iA=A_0+C, 
\end{equation*}
where $A_0 = H_\alpha - \lambda,$ and $C= - (\sigma +1)  |\phi |^{2\sigma } - \sigma  |\phi |^{2\sigma} {\mathcal T}$. \ It suffices to prove that $\sigma_{ess} (iA) \subset \R$. \ To this end, we remark the following facts.

\begin {itemize}

\item [-] Since $\phi\in H^1(\R) \subset L^{\infty}(\R)$, $C$ is a bounded operator.

\item [-] It is known that $\sigma_{ess}(A_0) = [-\lambda , + \infty ) \subset \R$.

\item [-] $C \left [A_0 +\lambda + 1  \right ]^{-1}$ is a compact operator; indeed, $\left [A_0 +\lambda + 1  \right ]^{-1}$ is an integral operator with kernel given by $K^0_\alpha (x,y;i) + \sum_{j=1}^4 K^j_\alpha (x,y;i)$. \ One can see, for e.g., that $K^j_\alpha (x,y;i)$ ($j=1,2,3,4$) and $ |\phi |^{2\sigma} K^0_\alpha (x,y;i)$ are bounded on $L^2 (\R^2 , dx \, dy)$. \ This implies that $C \left [A_0 +\lambda + 1  \right ]^{-1}$ is Hilbert-Schmidt. 

\end {itemize}

Then $\sigma_{ess} (A_0) = \sigma_{ess} (iA)$ by means of the Weyl criterion. 
\end{proof}

As for eigenvalues of $A$, there are finitely many eigenvalues at the exterior of the essential spectrum for $\hbar$ small. Indeed, $\lambda<0$ for $\hbar$ small. \ Our aim is now to conclude the following Proposition. 

\begin {proposition} \label{Proposition2}
Assume that $A$ has an eigenvalue $\lambda_m$ with $\Re \lambda_m >0$, and that for any $\varepsilon>0$, there exists $M>0$ such that 
\begin{equation} \label{Eq82}
\|e^{tA} v\|_{\mathbb{L}^2} 
\le M e^{(1+\varepsilon) (\Re\lambda_m) t} \|v\|_{\mathbb{L}^2},
\end{equation}
for any $v \in \mathbb{L}^2(\R)$ and for any $t \ge 0$. \ Then, there exists $\varepsilon_0>0$, such that for any $\delta>0$ there exist a time $T$ and an initial data $u_0 \in D(H_{\alpha})$ satisfying $\|u_0 -\phi \|_{H^1} <\delta,$ and $\inf_{\theta\in\R} \|u(T)-e^{i\theta} \phi \| \ge \varepsilon_0.$ 
\end {proposition}

Proposition \ref{Proposition2} means that the linearized instability implies the nonlinear instability. 

We may prove Proposition \ref{Proposition2} as in the proof of Theorem 6.1 of Part II of \cite{GSS} or in \cite{G}. \ Note that we have the Dirac measures in the equation and we do not expect that the solution is smooth as we have mentioned before, thus we make use rather of the time derivative, mimicking the proof of \cite{CoCoOh}, than of the way of \cite{GSS}. \ Here, for the sake of completeness, we give an outline of proof. 

\begin{proof}(Sketch of proof) 
Let $z_m$ be the associated eigenfunction to $\lambda_m$. \ Let $u_\delta(t)$ be the solution of (\ref{Eq79}) with initial data $u_\delta(0)=\phi+\delta z_m$. \ Since $\phi, z_m \in D(H_{\alpha}),$ $u_{\delta}(\cdot) \in C([0,T], D(H_{\alpha})) \cap C^1([0,T], L^2)$ for some $T>0$ (see Theorem 3.1 of \cite{AdNo}). \ Remark that $u_{\delta}(t)=e^{-i\lambda t}(\phi+v_{\delta}(t))$ with $v_{\delta}(t)$ satisfying (\ref{Eq81}) 
with $v_{\delta}(0)=\delta z_m$. \ $v_{\delta}(t)$ satisfies the following integral equations for any $t\in [0,T]$,
\begin{eqnarray*}
v_{\delta}(t) &=& \delta e^{\lambda_m} z_m 
+ \int_0^t e^{(\tau-t)A} F(v_{\delta}(\tau)) d\tau, \\
\partial_t v_{\delta}(t) &=& \lambda_m \delta e^{\lambda_m} z_m 
+e^{tA}F(\delta z_m)
+\int_0^t e^{(\tau-t)A} \partial_{\tau} F(v_{\delta}(\tau)) d\tau.
\end{eqnarray*}
Since we are in the one dimensional case, it is easy to estimate the nonlinear term $F(v)$ as follows. 
\begin{eqnarray*}
\|\partial_t F(v_{\delta}(t))\| &\le & C(\|v_{\delta}(t)\|_{H^1}+\|v_{\delta}(t)\|_{H^1}^{2\sigma}) \|\partial_t v_{\delta}(t)\|,\\ 
\|F(v_{\delta}(t))\| &\le & C_2 
(\|v_{\delta}(t)\|_{H^1}+\|v_{\delta}(t)\|_{H^1}^{2\sigma+1}). 
\end{eqnarray*}
Then, for some $C_0>0$ and for some $T_{\delta}$ when $\delta$ is sufficiently small, we may estimate   
\bee
\|v_{\delta}(t)\|_{H^1} +\|\partial_t v_{\delta}(t)\| \le  2 C_0 \delta e^{\lambda_m t},
\eee
for any $t\in [0,T_{\delta}]$. \ We apply this quantity $\|v_{\delta}(t)\|_{H^1} +\|\partial_t v_{\delta}(t)\|$ as $V_{\delta}(t)$ in Theorem 2 of \cite{CoCoOh}. \ We then repeat their arguments in \cite{CoCoOh} to get $\|v_{\delta}(T_{\delta})\| \ge (\delta/2) \|z_m\|$.
\end{proof}

We complete our whole arguments with a verification of the existence of an eigenvalue $\lambda_m$ satisfying (\ref{Eq82}). \ 
It follows from \cite{G} or Part I of \cite{GSS} that there exists a non-zero real eigenvalue of the linearized operator $A$, if (2) or (3) 
of Proposition \ref{Proposition1} in Section \ref{Sec5} hold. \ Let $\lambda_0$ be the maximal positive eigenvalue. \ Once we have proved the spectral mapping theorem $\sigma(e^{At})=e^{\sigma(A)t}$, the spectral radius of $e^{At}$ is $e^{\lambda_0 t}$. \ Thus we have (\ref{Eq82}) using Lemma 3 of \cite{SS}. \ This implies that we can take $\lambda_0$ as $\lambda_m$ in Proposition \ref{Proposition2}.

The spectral mapping theorem in fact follows from a resolvent estimate in Lemma \ref{Lemma11} below, combined with the arguments in \cite{GJLS}. 

\begin {lemma} \label{Lemma11} 
Let $z=a+i\tau$ with $a,\tau \in \R$ and $a \ne 0$. \ For $|\tau|$ sufficiently large, there exists a constant $C_a>0$, such that
\bee
\|(z-A)^{-1}\|_{\mathcal{L}(\mathbb{L}^2)} \le C_a.
\eee
\end {lemma} 

\begin {proof} [Proof of Lemma \ref {Lemma11}] We begin with some preparations. \ For fixed $z=a+i\tau$ with $a\in \R \setminus \{0\}$ 
and $\tau \in \R$, we write the operator $z-A$ as follows. 
\bee
z-A&=&M_z-B_{\lambda,\epsilon}=
\begin{pmatrix}
z & -H_{\alpha} \\
H_{\alpha} & z
\end{pmatrix}
+
\begin{pmatrix}
0 & \phi^{2\sigma}+\lambda \\ 
-\lambda -(2\sigma+1)\phi^{2\sigma} & 0
\end{pmatrix}
\\
&=&M_z[Id-M_z^{-1} B_{\lambda, \epsilon}]
\eee
Indeed, we see that $z \notin i\R$, therefore, by Remark \ref{Remark20}, the inverse of $H_{\alpha}^2 +z^2 = (H_{\alpha}-iz)(H_{\alpha} +iz)$ exists, thus the inverse of $M_z$ exists too. \ We can express $M_z^{-1}$ as follows. 
\bee
M_z^{-1}=
\begin{pmatrix}
z\{(H_{\alpha})^2+z^2\}^{-1} & H_{\alpha} \{(H_{\alpha})^2 + z^2\}^{-1} \\
-H_{\alpha} \{(H_{\alpha})^2 +z^2\}^{-1} & z \{(H_{\alpha})^2 +z^2\}^{-1} 
\end{pmatrix} 
\eee
We estimate now the inverse $M_z^{-1}$ by means of the following Lemma.

\begin {lemma} \label {Lemma12}
Let $a \ne 0.$ There exist $C_a$, $\tau_0 >0$ such that for any $z=a+i\tau$ with $|\tau| \ge \tau_0$, we have 
\bee
\|M_z^{-1}\|_{\mathcal{L}(\mathbb{L}^2)} \le \frac{C_a}{1+|\tau|}.
\eee
\end {lemma}

\begin {proof}[Proof of Lemma \ref {Lemma12}] We benefit from the explicit resolvent formula of $H_{\alpha}$  in Remark \ref{Remark20}. \ Let
\bee 
f_{\alpha}(x)=([H_{\alpha}-k^2]^{-1} h)(x)
\eee
and consider $k^2=iz=-\tau+i a$. \ First, we remark that  
\bee
f_0(x) = \int_{\R} 
K_0 (x,y;k) h (y) d y 
= \frac {i}{2k} \left ( e^{i k |\cdot |} \star h (\cdot ) \right ) (x)
\eee
may be estimated, by Young inequality, as follows.   
\bee
\|f_0 \| 
= \frac {1}{2 |k|} 
\left \| e^{ik|\cdot |} \star h (\cdot ) \right \|  
\le  \frac {1}{2 |k|} \left \| e^{ik|\cdot |}  \right \|_{L^1}  \| h (\cdot ) \|  
\le  \frac {C}{|k| |\Im k | } \|h \| 
\le \frac {C}{\sqrt {\tau } } \|h \|,  
\eee
since 
\bee
\sqrt {\tau + i a } = \sqrt {\tau } \sqrt {1+ \frac {ia}{\tau}} = \sqrt {\tau} + \frac {ia}{2\sqrt {\tau}} + \asy {\tau^{-3/2}}\, , \mbox { as } |\tau | \to \infty \, , 
\eee
and $\Im \sqrt {\tau + i a } \sim \frac 12 \frac {a}{\sqrt {\tau}}$ for $\tau \gg 1$. 

Next, we set  
\bee
f_\alpha^j (x) = \int_{\R} K_\alpha^j (x,y;k) h (y) d y \, ,\ \ j=1,2,3,4.
\eee
By this definition, $f_{\alpha}= f_0 + \sum_{j=1}^{4} f_{\alpha}^j$. \ Thus we estimate each term $f_{\alpha}^j$. For example,  
\bee
f_{\alpha}^1 (x) 
= \frac {\alpha (2k+i\alpha )}{2k \left ( (2k+i\alpha )^2 + \alpha^2 e^{i4ka} \right ) } 
e^{ik |x+a|} \int_{\R} e^{ik|y+a| } h (y) d y 
\eee
and then, for sufficiently large $|\tau|$,  
\bee
\|f_{\alpha}^1\| 
&\le & \frac {C}{|k|^2 |\Im k|} \left | \int_{\R} e^{ik|y+a| } h (y) d y \right | 
\le  \frac {C}{|k|^2 |\Im k|} \left \| e^{ik|\cdot +a| } h (\cdot ) \right \|_{L^1}   \\
&\le & \frac {C}{|k|^2 |\Im k|} \left \| e^{ik|\cdot +a| } \right \| 
\left \| h \right \| 
\le  \frac {C}{|k|^2 |\Im k|^2}  \left \| h \right \| 
 \le \frac {C}{|\tau|} \| \left \| h \right \|.
\eee
Similarly, the other terms $f_{\alpha}^j$, $j=2,3,4$, are estimated. \ Thus, it follows that for $\Im z = a$ fixed and $\Re z =-\tau $ large enough, then 
\bee
\left \| [H_\alpha - iz ]^{-1} h \right  \| \le \frac {1}{|\tau|} 
\| h \|.
\eee
since $H_\alpha$ is a self-adjoint operator. \ Therefore, decomposing $H_{\alpha}\{(H_{\alpha})^2 + z^2\}^{-1}$ as 
\bee
H_{\alpha}\{(H_{\alpha})^2 + z^2\}^{-1}=(H+iz)^{-1}+iz(H-iz)^{-1}(H+iz)^{-1},
\eee
we also obtain, for large $|\tau| \gg 1$, 
\bee
\|H_{\alpha}\{(H_{\alpha})^2 + z^2\}^{-1}\|_{\mathcal{L}(\mathbb{L}^2)} 
\le \frac{C_a}{1+|\tau|}.
\eee
Similarly, for large $|\tau|$,
\bee
\|z \{(H_{\alpha})^2 + z^2\}^{-1}\|_{\mathcal{L}(\mathbb{L}^2)} 
\le \frac{C_a}{1+|\tau|}. 
\eee
\end {proof}

We go back to the proof of Lemma \ref{Lemma11}. \ We put $T_z= M_z^{-1} B_{\varepsilon, \lambda}$, and we write entries of this operator $T_z$: 
\bee
T_z=
\begin{pmatrix}
H_{\alpha}\{(H_{\alpha})^2+z^2\}^{-1}(-\lambda-(2\sigma+1)\phi^{2\sigma}), 
& z \{(H_{\alpha})^2+z^2\}^{-1} (\phi^{2\sigma}+\lambda) \\
z \{(H_{\alpha})^2+z^2\}^{-1} (-\lambda-(2\sigma+1)\phi^{2\sigma}), 
& -H_{\alpha} \{(H_{\alpha})^2+z^2\}^{-1}(\phi^{2\sigma}+\lambda)
\end{pmatrix}
\eee
Since we are in one dimension, it follows that $\phi \in H^1(\R) \subset L^{\infty}(\R)$, thus we can estimate, for example, as 
\bee
\|H_{\alpha} \{(H_{\alpha})^2+z^2\}^{-1}(\phi^{2\sigma}+\lambda)
\|_{\mathcal{L}(\mathbb{L}^2)} 
\le C \|H_{\alpha}\{(H_{\alpha})^2+z^2\}^{-1}
\|_{\mathcal{L}(\mathbb{L}^2)}.
\eee
Therefore, combining with the above proof for Lemma \ref{Lemma11}, we have that for any $\tau$ with $|\tau|\ge \tau_0$, $\|T_z\|_{\mathcal{L}({\mathbb{L}^2})} \le 1/2$. \ This implies immediately for any $u\in \mathbb{L}^2$ 
\bee
\|(Id-T_z)u\|_{\mathbb{L}^2} \ge \|u\|_{\mathbb{L}^2}-\|T_z u \|_{\mathbb{L}^2} \ge (1/2) \|u\|_{\mathbb{L}^2},
\eee
that is, $Id-T_z$ is invertible for $|\tau| \ge \tau_0$. \ Then, finally, we get that for any $z=a+i\tau$ with $|\tau| \ge \tau_0$, $a \ne 0$, 
\begin{eqnarray*}
\|(z-A)^{-1}\|_{\mathcal{L}(\mathbb{L}^2)}
&=&\|(Id-T_z)^{-1} M_z^{-1}\|_{\mathcal{L}(\mathbb{L}^2)} \\
&\le& \|(Id-T_z)^{-1}\|_{\mathcal{L}(\mathbb{L}^2)} 
\|M_z^{-1}\|_{\mathcal{L}(\mathbb{L}^2)} \le 2C_a.
\end{eqnarray*}
The proof of Lemma \ref {Lemma11} is then completed. 
\end {proof}

Lastly, recall that the assumptions (2) or (3) of Proposition \ref{Proposition1} in Section \ref{Sec5} ensure the existence of a positive real eigenvalue of $A$. \ As we checked in Section \ref{Sec5}, the assumptions (2) or (3) of Proposition \ref{Proposition1} in Section \ref{Sec5} may be verified, for small $\hbar >0$, depending on $\sigma,$ $\eta,$ and the sort of stationary solution. \ Namely, Theorem \ref{Theorem3} in Section \ref{Sec5} is valid for Eq.(\ref{Eq79}).


\begin{thebibliography}{99}

\bibitem{AdNo} R.Adami and D.Noja, {\it Existence of dynamics for a 1-d NLS equation perturbed with a generalized point defect,} J. Phys. A: Math. Theor. {\bf 42}, 495302 (2009).  

\bibitem{Albeverio} S.Albeverio, F.Gesztesy, R.Hoegh-Krohn and H.Holden, {\it Solvable models in quantum mechanics,} AMS Chelsea publishing (2005)

\bibitem {BaSa} D.Bambusi and A.Sacchetti, {\it Exponential times in the one-dimensional Gross-Pitaevskii equation with multiple well potential}, Comm. Math. Phys. {\bf 275}, 1-36 (2007).

\bibitem{BeSh} F.A.Berezin and M.A.Shubin, {\it The Schr\"odinger equation,} Kluwer Academic publishers, (1991)

\bibitem {C} T.Cazenave, {\it Semilinear Schr\"odinger Equations}, Courant Lecture Notes in Mathematics (New York 2003).

\bibitem {Christian} J.M.Christian, G.S.McDonald, R.J.Potton, and P.Chamorro-Posada, {\it Helmholtz solitons in power-law optical materials}, Phys. Rev. A {\bf 76}, 033834 (2007).

\bibitem{CoCoOhIHP} M.Colin, T.Colin and M.Ohta, {\it 
Stability of solitary waves for a system of nonlinear Schr\"odinger equations with three wave interaction},  Ann. Inst. H. Poincar\'{e} Anal. Non Lin\'{e}aire {\bf 26}, 2211--2226 (2009).

\bibitem{CoCoOh} M.Colin, T.Colin and M.Ohta, {\it Instability of standing waves for a system of nonlinear Schr\"{o}dinger equations with three-wave interaction}, Funkc. Ekvac. {\bf 52}, 371-380 (2009).

\bibitem{CoPe} A.Comech and D.Pelinovsky, {\it Purely nonlinear instability of standing waves with minimal energy}, Comm. Pure Appl. Math. {\bf 56}, 1565-1607 (2003). 

\bibitem {CLLLR} M.Coste, T.Lajous-Loaeza, H.Lombardi and M.F.Roy, {\it Generalized Budan-Fourier theorem and virtual roots}, J. Compl. {\bf 21}, 479-486 (2005).

\bibitem{DeFe} M.Del Pino and P.L.Felmer, {\it Multi-peak bound states for nonlinear Schrodinger equations}, Ann. Inst. H. Poincar\'e - Analyse Non Lineaire {\bf 15}, 127-149 (1998). 

\bibitem{DiGa} L.Di Menza and C.Gallo, 
{\it The black solitons of one-dimensional NLS equations}, 
Nonlinearity. {\bf 20}, 461-496 (2007).  

\bibitem{FlWe} A.Floer and A.Weinstein, {\it Nonspreading wave packets for the cubic Schr\"odinger equation with a bounded potential}, J. Funct. Anal. {\bf 69}, 397-408 (1986). 

\bibitem{FuOhOz} R.Fukuizumi, M.Ohta and T.Ozawa, {\it Nonlinear Schr\"odinger equation with a point defect}, Ann. Inst. H. Poincar\'e - Analyse Non Lineaire {\bf 25}, 837-845 (2008).

\bibitem{FuOz} R.Fukuizumi and T.Ozawa, {\it Exponential decay of solutions to nonlinear elliptic equations with potentials}, Zeit. fur Ang. Math. und Phys. {\bf 56}, 1000-1011 (2005). 

\bibitem{GJLS} F.Gesztesy, C.K.R.T.Jones, Y.Latushkin and M.Stanislavova, {\it A spectral mapping theorem and invariant manifolds for nonlinear Schr\"odinger equations}, Indiana Univ. Math. J. {\bf 49}, 221-243 (2000).

\bibitem{GMS} V.Grecchi, A.Martinez and A.Sacchetti, {\it Destruction of the beating effect for a nonlinear Schr\"odinger equation,} Comm. Math. Phys. {\bf 227}, 191-209 (2002) 

\bibitem{G} M.Grillakis, {\it Linearized instability for nonlinear Schr\"odinger and Klein-Gordon equations}, Comm. Pure Appl. Math. {\bf 41}, 745-774 (1988). 

\bibitem{GSS} M.Grillakis, J.Shatah and W.Strauss, {\it Stability theory of solitary waves in the presence of symmetry I, II}, J. Funct. Anal. I: {\bf 74}, 160-197 (1987). II: {\bf 94}, 308-348 (1990).

\bibitem{H} B.Helffer, {\it Semi-classical Analysis for the Schr\"odinger operator and applications}, Lecture Note in Mathematics, 1336, Springer-Verlag (1980).

\bibitem {JaWe} R.K.Jackson and M.I.Weinstein, {\it Geometric Analysis of Bifurcation and Symmetry Breaking in a Gross-Pitaevskii Equation}, J. Stat. Phys. {\bf 116}, 881-905 (2004).

\bibitem {Kohler} T.K\"ohler, {\it Three-Body Problem in a Dilute Bose-Einstein Condensate}, Phys. Rev. Lett. {\bf 89}, 210404 (2002).

\bibitem {KKSW} E.W.Kirr, P.G.Kevrekidis, E.Shlizerman and M.I.Weinstein, {\it Symmetry-breaking bifurcation in nonlinear Schr\"odinger/Gross-Pitaevskii equations}, SIAM J. Math. Anal. {\bf   40}, 566-604 (2008).

\bibitem {KoSa} H.Kovarik and A.Sacchetti, {\it A nonlinear Schr\"{o}dinger equation with two symmetric point interactions in one dimension}, J. Phys. A: Math. Theor. {\bf 43}, 155205 (2010). 

\bibitem {Maeda} M.Maeda, {\it Stability of bound states of Hamiltonian PDEs in the degenerate cases}, Preprint. 

\bibitem {Mihalache} D.Mihalace, M.Bertolotti, and C.Sibilia, {\it Nonlinear wave propagation in planar structures}, Prog. Opt. {\bf 27}, 229 (1989).

\bibitem {Ohta} M.Ohta, {\it Instability of bound states for abstract nonlinear 
Schr\"odinger equations}, Preprint. arXiv:1010.1511v1

\bibitem {PitStr} L.Pitaevskii, and S.Stringari, {\it Bose-Einstein condensation}, (Claredon Press: Oxford 2003). 

\bibitem {Prasolov} V.V.Prasolov, {\it Polynomials}, Springer Verlag (Berlin 2001).

\bibitem{SS} J.Shatah and W. Strauss, {\it Spectral condition for instability},   
Comtemp. Math. {\bf 255}, 189-198 (2000).  

\bibitem {S} A.Sacchetti, {\it Nonlinear double well Schr\"odinger equations in the semiclassical limit}, J. Stat. Phys. {\bf 119}, 1347-1382 (2005).

\bibitem {Sacchetti2} A.Sacchetti, {\it Universal critical power for nonlinear Schr\"odinger equations with a symmetric double well potential}, Phys. Rev. Lett. {\bf 103}, 194101 (2009).

\bibitem {Smerzi} A.Smerzi, and A.Trombettoni, {\it Nonlinear tight-binding approximation for Bose-Einstein condensates in a lattice}, Phys. Rev. A {\bf 68}, 023613 (2003).

\bibitem {Snyder} A.W.Snyder, and D.J.Mitchell, {\it Spatial solitons of the power-law nonlinearity}, Opt. Lett. {\bf 18}, 101 (1993).

\bibitem {Zakharov} V.E.Zakharov, and V.S.Synakh, {\it The nature of self-focusing singularity}, Zh. Eksp. Teor. Fiz. {\bf 68}, 940 (1975) [Sov. Phys. JETP {\bf 41}, 465 (1975)].

\end{thebibliography}
\end {document}